\documentclass[12pt]{amsart}

\usepackage{epsfig}
\usepackage{graphics}
\usepackage{amsfonts}
\usepackage{amsmath}
\usepackage{latexsym}

\newenvironment{preuve}{\noindent {\it Proof}}{\hfill$\square$}

\newcommand\sk{\smallskip}
\newcommand\de{\delta}
\newcommand\ep{\varepsilon}
\newcommand\ph{\varphi}

\newcommand\zu{[0,1]}
\newcommand{\ds}{\displaystyle}
\newcommand{\ho}{H\"older }
\newcommand{\al}{\alpha }
\newcommand{\rr}{{\mathbb{R}}}
\newcommand{\RR}{{\mathbb{R}}}
\newcommand{\R}{{\mathbb{R}}}

\newcommand{\intot}{\ds\int_0^t}
\newcommand{\sm}{{s-}}
\newcommand{\um}{{u-}}
\newcommand{\tm}{{t-}}
\newcommand{\indiq}{{\bf 1}}
\newcommand{\ala}{\nonumber \\}
\newcommand{\e}{\varepsilon}
\newcommand{\nk}{{n_k}}
\newcommand{\cL}{{\mathcal L}}
\newcommand{\cG}{{\mathcal G}}
\newcommand{\cP}{{\mathcal P}}

\newcommand{\cS}{{\mathcal S}}
\newcommand{\cH}{{\mathcal H}}
\newcommand{\cJ}{{\mathcal J}}
\newcommand{\cF}{{\mathcal F}}
\newcommand{\cZ}{{\mathcal Z}}
\newcommand{\cB}{{\mathcal B}}

\newcommand{\E}{{\mathbb{E}}}

\newtheorem{theo}{Theorem}[]
\newtheorem{pro}[theo]{Proposition}

\newtheorem{remark}[theo]{Remark}
\newtheorem{lemma}[theo]{Lemma}
\newtheorem{defi}[theo]{Definition}

\begin{document}

\title{A pure jump Markov process with a random singularity  spectrum}
%\author{Julien {\sc Barral}\footnote{INRIA, Domaine de Voluceau Rocquencourt,
%78153 Le Chesnay Cedex, France. E-mail: {\tt julien.barral@inria.fr} }, 
%Nicolas {\sc Fournier}\footnote{
%Labo. d'Analyse et de Math. Appliqu\'ees,
%Facult\'e de Sciences et Technologies,
%61 av du Gal de Gaulle, 94010 Cr\'eteil
%Cedex, France. E-mails: {\tt nicolas.fournier@univ-paris12.fr}, 
%{\tt jaffard@univ-paris12.fr},  {\tt seuret@univ-paris12.fr}}, 
%St\'ephane {\sc Jaffard}$^\dag$\\
%and St\'ephane {\sc Seuret}$^\dag$}

\author{Julien Barral}
\author{Nicolas Fournier}
\author{St\'ephane Jaffard}
\author{St\'ephane Seuret}
\address{Julien Barral: INRIA, Domaine de Voluceau Rocquencourt,
78153 Le Chesnay Cedex, France.}
\email{julien.barral@inria.fr}
\address{Nicolas Fournier, St\'ephane Jaffard, St\'ephane Seuret:
Laboratoire d'Analyse et de Math\'ematiques Appliqu\'ees - CNRS UMR 8050 - Universit\'e Paris-Est - UFR   Sciences et Technologie - 
61, avenue du G\'en\'eral de Gaulle, 94010 Cr\'eteil Cedex, France}
\email{nicolas.fournier@univ-paris12.fr}
\email{jaffard@univ-paris12.fr}
\email{seuret@univ-paris12.fr}

\begin{abstract}
\noindent 
We construct a  non-decreasing 
pure jump Markov process, whose jump
measure heavily depends  on the values taken by the process.
We determine the singularity spectrum of this process, which
turns out to be random and to depend locally on the values taken by the process. The result relies on fine properties of the distribution of Poisson point processes and on  ubiquity theorems.
\end{abstract}



\maketitle

%%%%%%%%%%%%%%%%%%%%%%%%%%%%%%%
%%%%%%%%%%%%%%%%%%%%%%%%%%%%%%%
%%%%%%%%%%%%%%%%%%%%%%%%%%%%%%%
%%%%%%%%%%%%%%%%%%%%%%%%%%%%%%%
%%%%%%%%%%%%%%%%%%%%%%%%%%%%%%%
%%%%%%%%%%%%%%%%%%%%%%%%%%%%%%%
\section{Introduction }

Up to the mid-70s, the study of the H\"older regularity of the sample paths of stochastic processes was focused on  two main issues: the determination of their uniform modulus of continuity, and  the existence of an almost-everywhere pointwise modulus of continuity.  However, the first indications that their pointwise regularity could vary from point to point  in a  subtle  way,   appeared  in the works of Orey-Taylor \cite{ot} and Perkins \cite{perkins}, who showed that the
fast and  slow points of Brownian motion are located on random fractal sets. Furthermore, they determined the   Hausdorff dimensions of these sets.  Brownian motion, however, only displays  very slight changes in its modulus of continuity (which is modified  only by logarithmic corrections).  This is in sharp contrast with other types of processes, such as L\'evy processes for instance, whose modulus of continuity changes completely from point to point.  Let us  recall the  relevant definitions  related with  pointwise H\"older  regularity,  in this context.

\begin{defi}  \label{defpoint}  Let $f: \RR_+  \rightarrow  \RR $ 
be a { locally   bounded function},  $t_0 \in   \RR_+$ and  let $\al > 0$. 
The function  $f $ belongs to $ C^\alpha (t_0)$ if there exist $C>0$ and a polynomial  $P_{t_0}$ of degree less than  $\alpha$ such that 
for all $t$ in a neighborhood of $t_0$,
\[ |f(t) -P_{t_0}(t) | \leq C | t-t_0|^\alpha. \] 
The {H\"older exponent} of $f$ at $t_0$ is (here $\sup \emptyset =0$)
\[ h_f(t_0) = \sup \{ \alpha > 0 : \hspace{2mm} f \in C^\alpha (t_0)\} . \]  
\end{defi}

The   level sets of the pointwise exponent of  the  H\"older exponent are  called  the  {\em iso-H\"older sets} of $f$,  and  defined, for $h\geq 0$, by 
\[ E_f(h)  = \{ t \geq 0: \hspace{2mm}  h_f(t) = h \} .  \]
The corresponding   notion of  ``multifractal function'' was  put into light by Frisch and Parisi \cite{FP}, who introduced the definition of the {\em spectrum of singularities} of a function $f$.

\begin{defi}   Let $f: \RR_+  \rightarrow  \RR $  be a 
locally   bounded function. 
The spectrum of singularities of $f$ is the function $D_f$ defined,
for $h\geq 0$, by 
\[ D_f (h) = \dim_H (E_f(h) ) .\]
We also define, for any open set 
$A \subset \RR_+$, $D_f (A,h) = \dim_H (E_f(h)\cap A)$.
\end{defi}

The definition of the Hausdorff dimension can be found in
Falconer \cite{f} for instance (by convention  $\dim_H \emptyset= -\infty$). The singularity spectrum of $f$  describes  the geometric repartition of its H\"older singularities, and encapsulates a geometric information which is usually more meaningful  than the H\"older exponent. 

Following the way opened by Frisch and Parisi, the spectrum of singularities of large classes of stochastic processes (or random measures, in which case an appropriate notion of H\"older pointwise regularity for measures is used) have been determined.  Most examples of stochastic processes $f$ which have been studied  display the following remarkable features.

\begin{itemize}
\item Though the iso-H\"older sets are random, the spectrum of singularities is deterministic: for some deterministic function $\Delta$, a.s., for all
$h\geq 0$, $D_f (h) = \Delta (h)$.
%\begin{equation} \label{specdet} \mbox{almost surely, } 
%\hspace{3mm}  \forall h\geq 0,\hspace{8mm}  D_f (h) = \Delta (h). 
%\end{equation}
\item The spectrum of singularities of $f$ is {\em homogeneous}: a.s.,
for any nonempty open subset $A \subset \RR_+$, for all $h\geq 0$, 
$D_f(A,h)=\Delta(h)$.
\end{itemize}

Though it is easy to construct artificial {\it ad hoc} processes that do not satisfy these properties, it is remarkable that  many ``natural'' processes of very different kind  follow this rule:   L\'evy processes \cite{j}, L\'evy processes in multifractal time \cite{BS1}, fractional Brownian motions,
random self-similar measures and random Gibbs measures \cite{BMP}, Mandelbrot cascades \cite{BA1}, Poisson cascades \cite{BA2}, among many other examples.
See however \cite{durand} where Durand constructed  a  counterexample   whose wavelet coefficients are defined using Markov trees.

%Deterministic homogeneous spectra are, for instance, displayed by Fractional Brownian Motions (whose spectra, are however degenerate) or L\'evy processes.  

%
%Note that these processes have  stationary increments.\\ \\
%{\em  { \bf Jaffard: } Que peut-on dire de g\'en\'eral sur le spectre d'un processus \`a accroissements stationaires?   Il est clair que, si le spectre est localement d\'eterministe,  alors,  il est homog\`ene, car, du fait de la stationarit\'e, il est le m\^eme sur des intervalles de longueur \'egale, et donc, en recollant les morceaux,  (qui peuvent \'eventuellement se recouvrir un peu), il est aussi le m\^eme sur deux intervalles qcq; mais  il  ne me semble pas \'evident de montrer que: accroissements stationnaires implique d\'eterministe, m\^eme si j'ai du mal a envisager que ce soit faux...  }\\? \\?
%{\em  { \bf Seuret: } C'est en effet une question, je pense que cela doit faire l'objet d'une courte remarque (du style "it would be interesting to investigate the relationship between the stationarity of the increments and the determinism of the local spectrum", pas beaucoup plus sinon on s'embarque dans un truc peut-etre dur (je ne sais pas a priori).}\\?

%Therefore, it is  natural to  investigate what happens when the assumption of stationarity of the increments is dropped.  

In this paper we will investigate  the regularity properties  of some Markov processes.  Our purpose at this stage is not to obtain results in the most  general form, but rather to consider some specific examples, and check that  such processes indeed display  a random spectrum, which is not homogeneous.  
Note that, until now, the only  Markov processes which have    been analyzed from the multifractal standpoint  are the  L\'evy processes. 
%Hence  the study of  local regularity properties  of  Markov processes, 
%which  are not  L\'evy processes, is  new.  \\

We now introduce a  new notion, the  {\em  local spectrum}, which  is tailored to the study of functions with non-homogeneous spectra. 

%and will therefore  provide the right setting in order to  state  the multifractality properties of the process $M$ in the sharpest  form. 

\begin{defi}  \label{defpoint2}  Let $f: \RR_+  \rightarrow  \RR $ be a { locally   bounded function}, $t_0 \in   \RR_+$ and  let  $(V_n)_{n\geq 1}$ 
be a basis of neighborhoods of $t_0$.  The local spectrum of $f$ at $t_0$ is  the function
\begin{equation*} %\label{equivlocspec}  
\hbox{ for all $h\geq 0$, } 
D_f (t_0, h) = \lim_{n\rightarrow \infty} D_f (V_n, h). 
\end{equation*}
\end{defi}

By monotonicity    (if $A \subset B$, then $ D_f (A, h) \leq  D_f (B, h)$), the limit exists and it  is independent of the particular basis chosen. 
Clearly, a  function has a homogeneous spectrum if and only if 
for all $h\geq 0$,  $D_f (t, h)$ 
is independent of $t\geq 0$.
The local spectrum   
allows one to recover  the  spectrum of all possible restrictions of $f$
on an open interval.

\begin{lemma}\label{spectresup}
Let $f: \RR_+ \mapsto \RR$ be a locally bounded function.
Then for any open interval $I=(a,b)\subset \RR_+$, for any $h\geq 0$, we have 
$  D_f (I,h) = \sup_{ t\in I} D_f (t, h)$.
\end{lemma}  

\begin{proof} Let thus $h\geq 0$ be fixed.
First, it is obvious that for any $t\in I$, $D_f (t, h)\leq D_f(I,h)$,
since for $(V_n)_{n\geq 1}$ a basis of neighborhoods of $t$, $V_n\subset I$ for
$n$ large enough. Next, set $\delta=D_f(I,h)$, and consider $\e>0$.
We want to find $t \in I$ such that for all neighborhood $V$ of $t$,
$D_f(V,h) \geq \delta -\e$. Assume by contradiction
that for any $t \in  I$, there is a neighborhood
$V_t$ of $t$ such that $D_f(V_t,h) < \delta -\e$. 
One easily deduces that for any compact $K\subset I$, 
$D_f(K,h)\leq \delta-\e$ (use a finite covering of $K$ by the $V_t$'s).
This would of course imply that $D_f(I,h)\leq \delta-\e<\delta$.
\end{proof}
%\begin{remark}
% Clearly, a  function has a homogeneous spectrum if and only if  $ D_f (x, H)$ is independent of $x$. \sk
%Moreover, one easily checks that, in the right-hand side of (\ref{spectresup}),  the supremum can be taken on an arbitrary dense  (and, in particular, countable) set.
%\end{remark}

Let us now recall the multifractal  nature of  L\'evy processes without Brownian component.  In that case, the spectrum of singularities only depends on one parameter, the { \em lower index} of Blumenthal and Getoor,  which quantifies the ``density'' of  small jumps, and is defined, 
for  any non-negative  measure   $\nu$ on $\rr$ satisfying $\int_{-1}^1 u^2\nu(du)<\infty$,
by \begin{equation*} 
\beta_\nu := \inf \left\{ \alpha\geq 0: \int_{-1}^1 |u|^\alpha \nu(du)<\infty
\right\} . 
\end{equation*}
Note that the integrability condition implies that $\beta_\nu\in [0,2]$.
The following result of  \cite{j} yields the spectrum of singularities of  such 
L\'evy processes.

%%%%%%%%%%%%%%%%
\begin{theo}
\label{thjaff} Let $(X_t)_{t\geq 0}$ be a L\'evy process without Brownian 
component, with L\'evy measure $\nu$. 
If $\beta_{\nu }\in(0,2)$, then 
the singularity spectrum 
of $X$ is homogeneous and deterministic:
a.s., for all $t\geq 0$, for all $h\geq 0$,
\begin{equation*}
%\label{speclevy}   %\hspace{-2cm}  
D_X(t,h)=D_X(h)= \begin{cases} h \cdot \beta_{\nu }    &  \mbox{ if } \  0 \leq  h\leq 1/\beta_{\nu} , \\  - \infty &  \mbox{ if } \ h > 1/\beta_{\nu }.\end{cases}
\end{equation*}
\end{theo}
%%%%%%%%%%%%%%%%%

Let us make a few  observations.
It is not stated explicitly in  \cite{j} that the spectrum of a L\'evy process is homogeneous, but it is  a direct consequence of  the proof. Indeed, the spectrum is computed on a arbitrary interval, and the stationarity 
of the increments implies  that it does not depend on the particular location of this interval. 
Although a L\'evy process  is random, its spectrum is almost surely deterministic. As explained above, this is the  situation usually met
when performing the multifractal analysis of random processes or random measures possessing either stationarity or  scaling invariance properties.

The purpose of this paper is to  investigate how these results are modified when the stationarity assumption is dropped. We will give    examples of  Markov processes whose singularity spectra  are non-homogeneous and random.

\section{Statement of  the main result}
%\section{Statement of results}

%\setcounter{equation}{0}

A quite general class of one-dimensional Markov processes consists of
stochastic differential equations (S.D.E.) with jumps, 
see Ethier-Kurtz \cite{ek}, Ikeda-Watanabe \cite{iw}.
Since the Brownian motion is mono-fractal, the Brownian part of such
a process will not be very relevant.
Thus in order to avoid technical difficulties,
we consider
a jumping S.D.E. without Brownian and drift part, starting e.g. from $0$,
and with jump measure $\nu(y,du)$ (meaning that 
when located at $y$, the process 
jumps to $y+u$ at rate $\nu(y,du)$).
Again to make the study as simple as possible, we will
assume that the process has finite variations, and even that it is increasing 
(that is, $\nu(y,(-\infty,0))=0$ for all $y \in \rr$). 
Classically, a necessary condition for the process to be well-defined is that
$\int_0^\infty u \,\nu(y,du)<\infty$.

\sk

If ${\nu }$ is chosen so that the index $ \beta_{{\nu }(y,.)}$ is constant with respect to $y$,  then we expect that $ D_M(t,h)$ will be deterministic and independent of $t$.
We thus impose that the index of the jump measure
depends on the  value
$y$ of the process. The most natural example of such a situation 
consists in choosing
\begin{equation*}%\label{defnu}
{\nu_\gamma}(y,du):= \gamma(y) u^{-1-\gamma(y)} \indiq_{[0,1]}(u)du,
\end{equation*}
for some function $\gamma: \rr \mapsto (0,1)$.  The lower exponent of this family of measures is 
$$
\forall \, y\geq 0, \ \ \\ \ \ \  \beta_{{\nu_\gamma}(y,.)}=\gamma(y).
$$

We will make the following assumption

\begin{equation*}
%\label{defh}
({ \mathcal H} ) \hspace{5mm} \begin{cases}  \mbox{There exists $\e>0$  such that $\gamma:[0,\infty)  \longmapsto [\ep,1-\ep]$ }
\\\mbox{ is a Lipschitz-continuous strictly increasing  function.}\end{cases}
\end{equation*}
%%%%%%%%%%%%%%%%%%%%%%%%%%%%%%%%%%%%%

We impose a monotonicity condition for simplicity. 
The global Lipschitz condition could be slightly relaxed, as well as
the uniform bounds.

%%%%%%%%%%%%%%%%%%%%%%%%%%%%%%%%%%%%%
\begin{pro}\label{exi}
Assume  that $({ \mathcal H} )$ holds.  There exists a strong Markov process
$M=(M_t)_{t\geq 0}$   starting from $0$,   increasing and c\`adl\`ag,
and with generator $\cL$
defined for all $y \in [0,\infty)$ and for any function 
$\phi: [0,\infty) \mapsto \rr$ 
Lipschitz-continuous by
\begin{equation}\label{giok}
\cL \phi(y) = \int_0^1 [\phi(y+u)-\phi(y)] 
{\nu_\gamma}(y,du).
\end{equation}
Almost surely, this process is  continuous except on  a countable number of jump times. We denote
by $\cJ $ the set of its jump times, that is $\cJ=\{t >0: \Delta M(t)\ne 0\}$.
Finally, $\cJ $ is dense in $[0,\infty)$.
\end{pro}
%%%%%%%%%%%%%%%%%%%%%%%%%%%%%%%%%%%%%

Here and below, $\Delta M_t=M_t - M_{t-}$, 
where $M_{t-} = \displaystyle\lim_{s\to t, \, s<t} M_s$.
Proposition \ref{exi} will be checked
in Section \ref{poisson}, by using a Poisson 
S.D.E. 
This representation of $M$ will be useful for its local regularity 
analysis in the next sections.

%

%As indicated above, $M$ has almost surely increasing, purely discontinuous sample paths.  Moreover, as will be explained soon, its  jump times   are related to a Poisson point process. 

%
%More precisely, there is a random set of couples 
%\[ \mathcal{P}= \{(T_n,\lambda_n)\}_{n\geq 1} \ \subset \zu\times \R^+,  \]
%which has the same law as  a Poisson point process of intensity 
%$$
% \Lambda(s,z)={\bf 1}_{\zu\times \R^+}(s,z)ds \frac{dz}{z},
%$$
%such that the jump times of $M$ are exactly the first coordinates of the Poisson points, i.e.
%$$\cJ = \{s\in \zu : \, \exists z\in \R^+ \mbox{ such that }(s,z)\in \mathcal{P}\}.$$
%Hence the size of the jump of $M$ at $T_n$, where $(T_n,\lambda_n)\in \mathcal{P}$, is exactly
%$(1+\lambda_n)^{-1/M_{T_n-}}$.

\smallskip

The following  theorem summarizes multifractal  features of $M$. 

%The two following  theorems   summarize the H\"older regularity  properties  of $M$, and its multifractal  features. 

%%%%%%%%%%%%%%%%%%%%%%%%%%%%%%%%%%%%%
%\begin{theo}\label{mr}
%Let $M$ be  the process $M$ constructed  in Proposition \ref{exi}. Then, the following properties hold  a.s. : 
%\begin{enumerate}
%\item
%  $E_M(1/\gamma(0))= \{0\} $ and $E_M(h)  = \emptyset $ as soon as   $h > 1/\gamma(0)$.
% 
% \medskip
%  
%\item
%For every   $t\in [0,1]$, one has $h_M(t) \leq  1/\gamma(M_t) $, and  \begin{eqnarray}
%\label{fms0}
%\mbox{for Lebesgue-almost every   $t\in [0,1]$, } &&h_M(t) = 1/\gamma(M_t) .
%\end {eqnarray}

%\medskip
%\item
%Finally, for every $\kappa\in(0,1]$, 
%\begin{eqnarray}
%\label{fms1}
% \dim_H \Big\{ t \in [0,1]: \; h_M(t)=\kappa/\gamma(M_t)\Big\} = \kappa .
% \end{eqnarray}
%\end{enumerate}
%\end{theo}
%%%%%%%%%%%%%%%%%%%%%%%%%%%%%%%%%%%%%

%Actually, we will explicitely compute the exponent of $M$ at every $t$. The value of this exponent depends on some approximation rate of $t$ by a point Poisson process.

%%%%%%%%%%%%%%%%%%%%%%%%%%%%%%%%%%%%%
\begin{theo}\label{mr1}
Assume $(\cH)$ and consider the process $M$ 
constructed  in Proposition \ref{exi}. 
Then, the following properties hold  almost surely.  
\begin{enumerate}
\item[(i)]
For every $t\in (0,\infty)\backslash \cJ$, 
the local spectrum of $M$ at $t$  is  given by 
\begin{equation}\label{localspecM}
D_M(t,h)  = \begin{cases} h \cdot  \gamma(M_t)     &  \mbox{ if } \  0\leq h\leq 1/\gamma(M_t)  , \\  
- \infty &  \mbox{ if } \ h > 1/\gamma(M_t),\end{cases}
\end{equation}
while for $t\in \cJ$, 
\begin{equation}\label{localspecM2}
D_M(t,h)  = \begin{cases} h \cdot \gamma(M_t)      &  \mbox{ if } \  0\leq h<1/\gamma(M_t)  , \\ 
h \cdot  \gamma(M_{t-})    &  \mbox{ if } \  h\in [1/\gamma(M_t), 1/\gamma(M_{t-})]  ,\\ - \infty &  \mbox{ if } \ h > 1/\gamma(M_{t-}).\end{cases}
\end{equation}
\item[(ii)]
The spectrum of $M$ on any interval $I=(a,b)\subset (0,+\infty)$  is  
\begin{eqnarray}
\label{mfs1} \forall  h \geq 0, \ \ \  D_M(I,h) & =&          \sup\Big \{h\cdot  \gamma(M_t): \, t\in I, \ h\cdot  \gamma(M_t) <1\Big\} \\
\label{mfs1.5} & =&  \sup\Big \{h\cdot  \gamma(M_\sm): \, s \in \cJ\cap I, \ h\cdot  \gamma(M_\sm) <1\Big\} .\end{eqnarray}
In \eqref{mfs1} and \eqref{mfs1.5}, we adopt the convention that $\sup \emptyset =-\infty$.
\end{enumerate}
\end{theo}
%%%%%%%%%%%%%%%%%%%%%%%%%%%%%%%%%%%%%

%%%%%%%%%%%%%%%%%%%%%%%%%%%%%%%%%%%%%
\begin{figure}[h]
\label{fig1}
\includegraphics[width=7.8cm,height = 4cm]{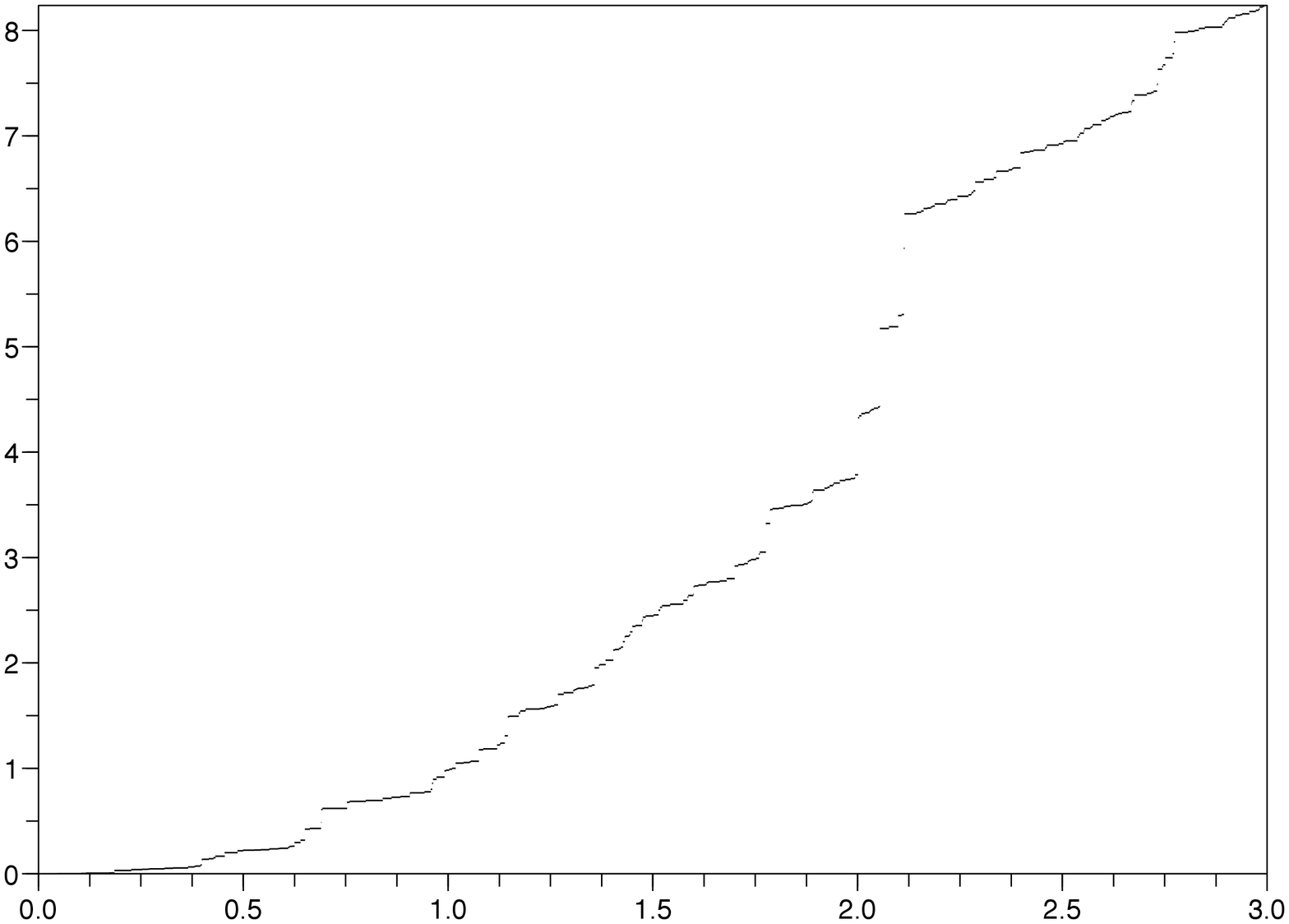}  
\includegraphics[width=7.8cm,height = 4cm]{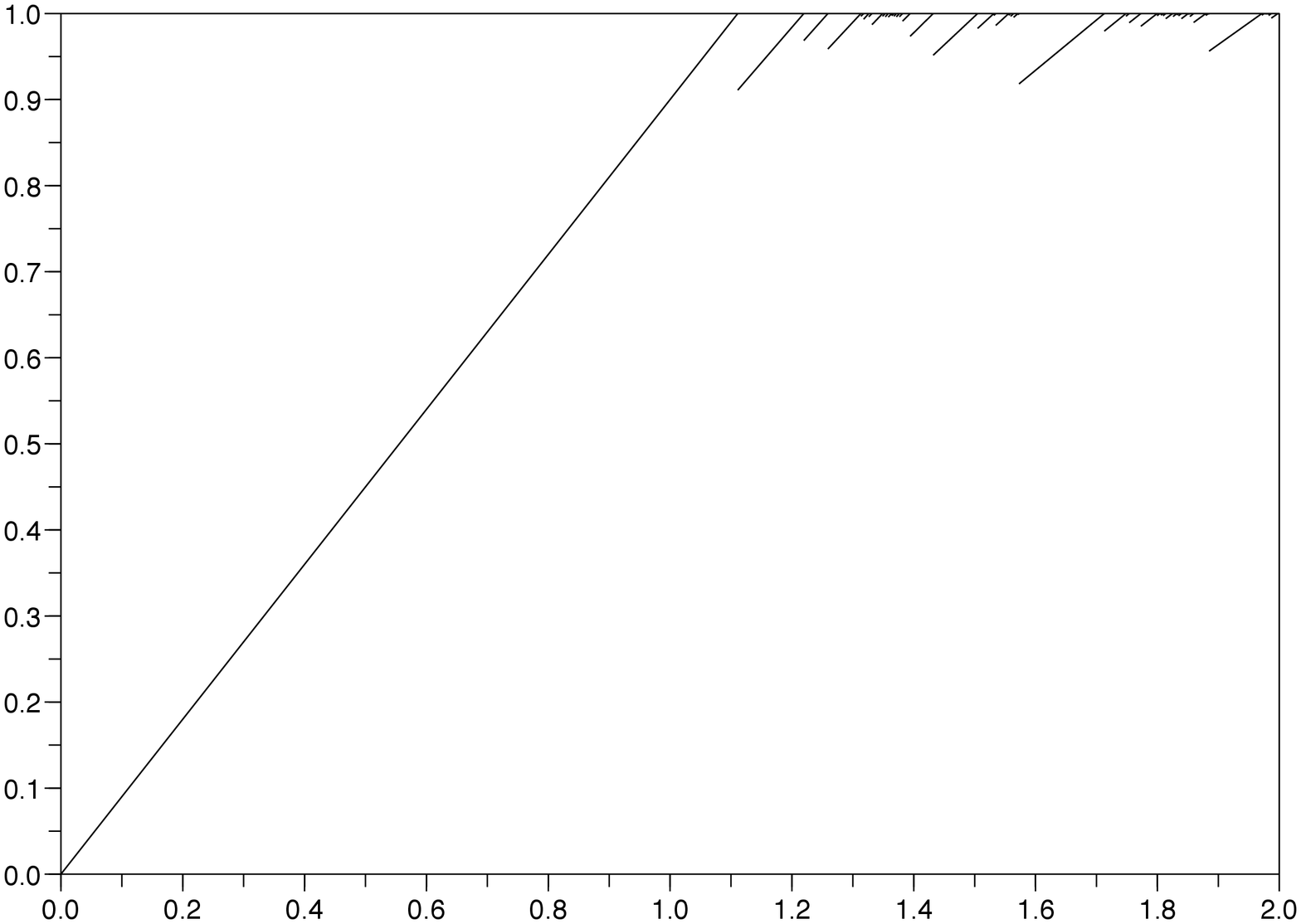}  
\includegraphics[width=7.8cm,height = 4cm]{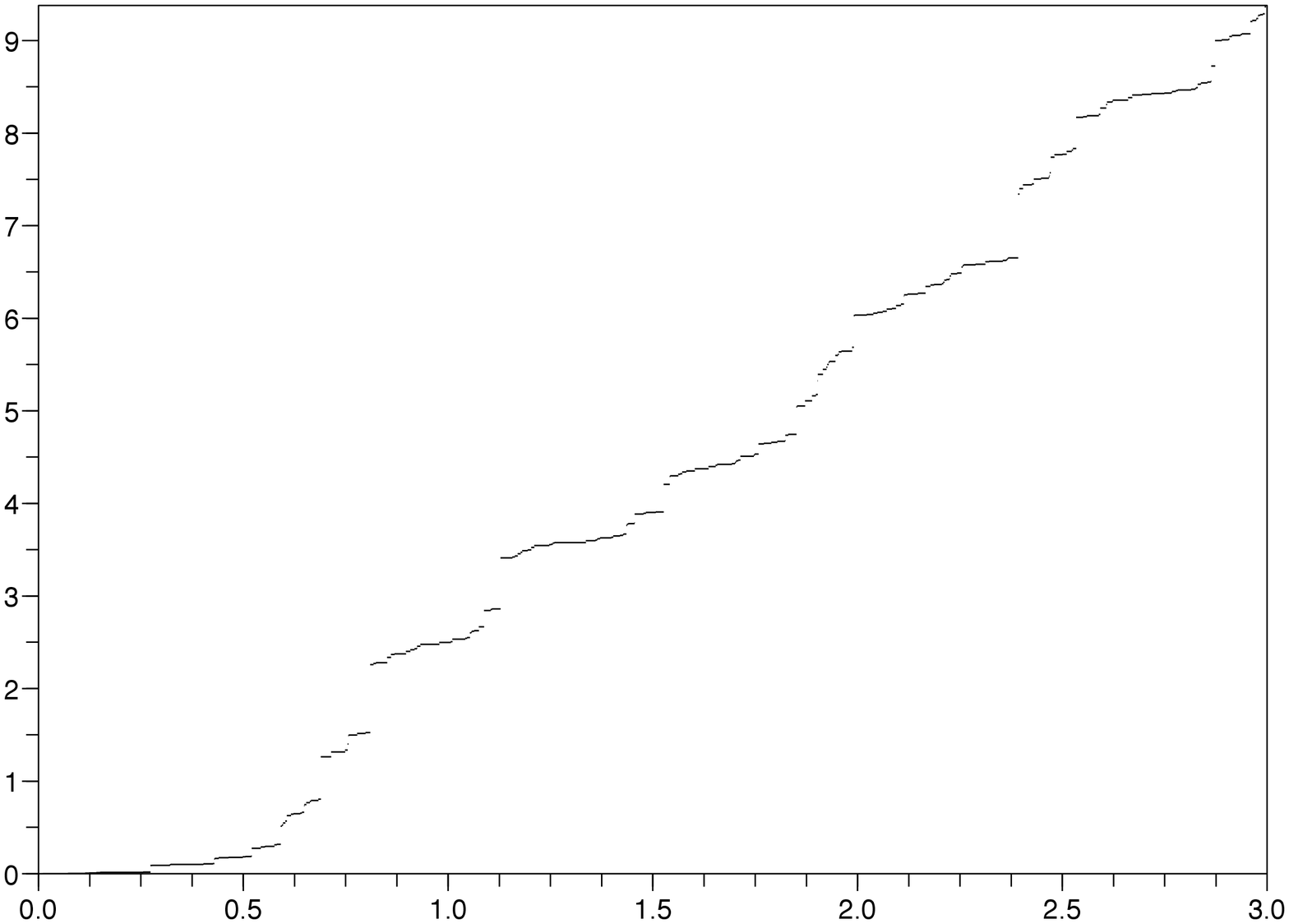}  
\includegraphics[width=7.8cm,height = 4cm]{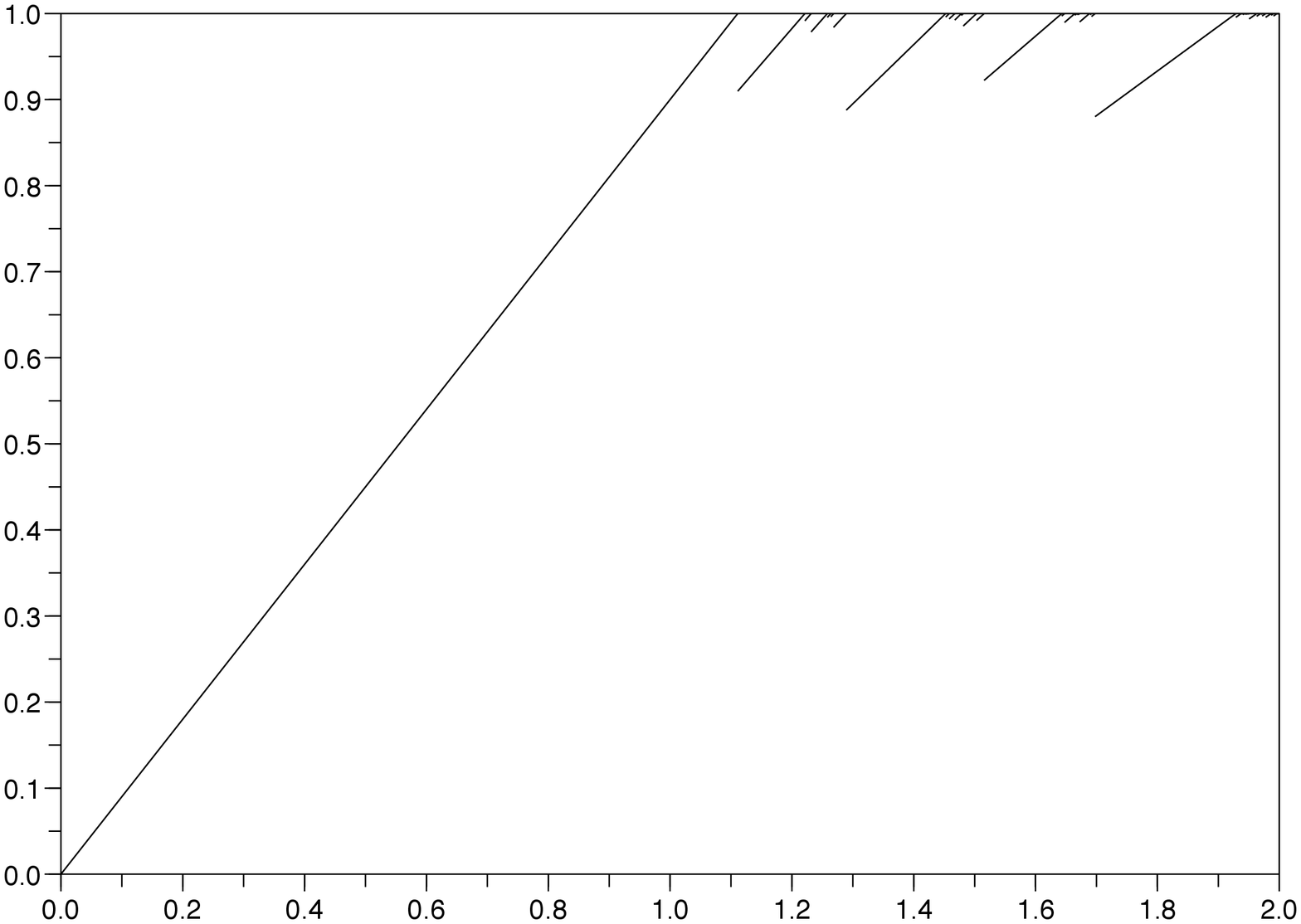} 
\caption{Two sample paths of the stochastic process $M$ built using the function $\gamma(y):=\min(1/2+y/4,0.9)$. On the right hand-side are plotted the 
theoretical spectra $D_M([0,3],.)$.}
\end{figure}
%%%%%%%%%%%%%%%%%%%%%%%%%%%%%%%%%%%%%%%%%%

\begin{remark} \label{enough}
To prove Theorem \ref{mr1}, it is enough to show \eqref{mfs1.5}.
\end{remark}
\begin{proof}
The equality between \eqref{mfs1} and  \eqref{mfs1.5} follows from the 
fact that the $\cJ$ is dense in $\rr_+$, and that $t\mapsto \gamma(M_t)$
is c\`adl\`ag on $I$.

\sk
Next, point (i) follows from \eqref{mfs1} applied to 
$I_n=(t-1/n,t+1/n)$ by taking the limit $n\to \infty$
(recall Definition \ref{defpoint2}). Let us for example assume that 
$t\in \cJ$.

\sk
\noindent $\bullet$ If $h>1/\gamma(M_{t-})$, then for $n$ large enough, 
$h \cdot \gamma(M_s)\geq h  \cdot \gamma(M_{t-1/n})\geq 1$ for all $s\in I_n$, 
whence $D_M(I_n,h)=-\infty$ (we use here that $s\mapsto \gamma(M_{s-})$ 
is non-decreasing). Thus $D_M(t,h)=-\infty$. 

\sk
\noindent $\bullet$ If $h<1/\gamma(M_t)$, we get from (\ref{mfs1}) that 
$h \cdot \gamma(M_t)\leq D_M(I_n,h)\leq h \cdot  \sup_{[t-1/n,t+1/n]}\gamma(M_s)=
h \cdot \gamma(M_{t+1/n})$. Since $s\mapsto \gamma(M_s)$ is right continuous, 
$D_M(t,h)= \lim_n  D_M(I_n,h)= h \cdot \gamma(M_t)$.

\sk
\noindent $\bullet$ If $h \in [1/\gamma(M_{t}),1/\gamma(M_{t-})]$, then for all
$s\geq t$, $h \cdot \gamma(M_s)\geq 1$, while clearly $h \cdot \gamma(M_{t-1/n})<1$. 
Hence we deduce from (\ref{mfs1}) that
$h \cdot \gamma(M_{t-1/n})\leq D_M(I_n,h)\leq h \cdot  \sup_{[t-1/n,t)}\gamma(M_s)=
h \cdot \gamma(M_{t-})$. Finally, $D_M(t,h)= \lim_n  D_M(I_n,h)= h \cdot \gamma(M_{t-})$.
\end{proof}

\sk
Formula \eqref{mfs1.5} is better understood when plotted: 
for every $s\in I \cap \cJ$, plot a segment whose endpoints are $(0,0)$ and  
$(1/\gamma(M_{s-}),1)$ (open on the right), and take the supremum to get 
$D_M(I,.)$. Sample paths of the process $M$ and their associated spectra are given in Figure 1.

%%%%%%%%%%%%%%%%%%%%%%%%%%
%%%%%%%%%%%%%%%%%%%%%%%%%%%%%%%%%%%%%%%%%%
%\begin{center}
%\begin{figure}[h] 
% \label{fig1}
% \includegraphics[width=7.8cm,height = 6.5cm]{markov4} 
%\caption{Typical singularity spectrum of a sample path of the Markov process $M$.}
%\end{figure}
%\end{center}
%%%%%%%%%%%%%%%%%%%%%%%%%%%%%%%%%%%%%%%%%%

The formulae giving the local and global spectra are based on the computation of the pointwise \ho exponents at all times $t$. 
We will in particular prove (see 
Theorem \ref{deltat} and Proposition \ref{specbiz})
the following properties: a.s.,
\begin{eqnarray*}
\mbox{for   every $t\geq 0$, } \ &&h_M(t) \leq 1/\gamma(M_t),\\
\mbox{for Lebesgue-almost every $t$, }&& h_M(t)= 1/\gamma(M_t),\\
\mbox{for every $\kappa\in (0,1)$, }&& 
\dim_H \{t\geq 0: \; h_M(t)= \kappa/\gamma(M_t) \}=\kappa.
\end{eqnarray*}

%\smallskip

It is worth emphasizing that, as expected from the construction of the process $M$, the local spectrum \eqref{localspecM} at any point $t>0$ essentially
coincides with that of a stable L\'evy subordinator of index $\gamma(M_t)$. This local comparison  will be   strengthened in Section~\ref{Tangent}, where we prove the existence of tangent processes for $M$ (which are L\'evy stable subordinators).

\sk

The proof of Theorem \ref{mr1}  requires  a so-called {\em ubiquity theorem}. Ubiquity theory deals with the search for lower bounds for the Hausdorff dimensions of limsup sets, and is classically required when performing the multifractal analysis of stochastic processes or (random or deterministic)  measures with jumps \cite{j,j2,BS1,BS2}. For our  Markov process $M$,  the ubiquity theorem needed here is the "localized ubiquity" theorem  recently proved in \cite{BS3}.
In order to apply this result,  we need to establish fine properties on   the distribution   of Poisson point processes (see Section \ref{properties}).

\begin{remark}
It follows from Theorem \ref{mr1} that
for all $s\in \cJ$, all $h\in (1/\gamma(M_s),1/\gamma(M_\sm)]$,  
$D_M(h) = h \cdot \gamma(M_\sm)  $. Thus the 
spectrum $D_M$ is a straight line on all segments of the form 
$(1/\gamma(M_s),1/\gamma(M_\sm)]$, $s\in \cJ$.
Nonetheless, observe that  the spectrum we obtain, 
when viewed as a map from $\R_+$ to $\R_+$,  is very irregular,
and certainly multifractal itself. 
This is in sharp contrast with  the spectra usually obtained, which are most of the time concave or (piecewise) real-analytic.
\end{remark}
\begin{remark}
Random processes with random singularity exponents of the most general form have already been constructed in \cite{AYAJAFFTAQ}, but there is no direct relationship with having a random singularity spectrum. An example of stochastic process $X$ (built using wavelet coefficients) with a random singularity spectrum is provided by \cite{durand}, but there  $D_X(h)$ is random for at most two values of $h$. 
\end{remark}

%\begin{remark}
%  The process $M$ will be shown to satisfy a stochastic differential equation (see formula \eqref{sde}). This is the first stochastic process defined through a non-trivial integral equation for which one is able to prove that it enjoys  multifractal properties.  
%\end{remark}

\begin{remark}
Of course we hope that Theorem \ref{mr1}, which concerns a specific and simple
process, might have extensions to more general Markov processes $M=(M_t)_{t\in[0,1]}$. In particular, this is certainly the case if we keep the same measures ${\nu_\gamma}$  and if we drop the monotonicity assumption on the  Lipschitz-continuous 
function $\gamma$.
\end{remark}

The organization of the paper is the following. In Section \ref{poisson},
Proposition \ref{exi} is proved. We also introduce a suitable coupling
of $M$ with a family of L\'evy processes.
In Section \ref{locreg},
we introduce a family of (random) subsets
of $[0,1]$ on which we control the regularity
of $M$. We conclude the proof of Theorem \ref{mr1} in Sections
\ref{low} and \ref{properties}.
Finally, in Section \ref{Tangent}, we show the existence of tangent processes 
for $M$.

\sk

{\bf In the whole paper, we assume that  $ ({ \mathcal H} )$ is satisfied.
We will restrict our study to the time interval $[0,1]$, 
which of course suffices.}

%%%%%%%%%%%%%%%%%%%%%%%%%%%%%%%%%%%%%
%%%%%%%%%%%%%%%%%%%%%%%%%%%%%%%%%%%%%
%%%%%%%%%%%%%%%%%%%%%%%%%%%%%%%%%%%%%
%%%%%%%%%%%%%%%%%%%%%%%%%%%%%%%%%%%%%
%%%%%%%%%%%%%%%%%%%%%%%%%%%%%%%%%%%%%
%%%%%%%%%%%%%%%%%%%%%%%%%%%%%%%%%%%%%
\section{Poisson representation of the process}\label{poisson}

First of all, we observe that using the substitution
$u=(1+z)^{-1/\gamma(y)}$ in (\ref{giok}) (for each fixed $y$),
we may rewrite, for any $y\in [0,\infty)$, for any
$\phi:[0,\infty) \mapsto \rr$ Lipschitz-continuous,
\begin{equation}\label{gifinal}
\cL \phi(y) = \int_0^\infty [\phi(y+G(\gamma(y),z))-\phi(y)] dz
\end{equation}
where
\begin{equation*}%\label{dfG}
G(\beta,z):=(1+z)^{-1/\beta}.
\end{equation*}

We recall the following representation of the Poisson measures
we are going to use.

%%%%%%%%%%%%%%%%%%%%%%%%%%%%%%%%%%%%%
\begin{remark}\label{spoisson}
Let $(Y_n)_{n\geq 1}$ be a sequence of independent exponential random variables with parameter $1$. Let $(T_n)_{n\geq 1}$ be a sequence of independent $[0,1]$-valued uniformly-distributed random variables, and assume that $(Y_n)_{n\geq 1}$ and $(T_n)_{n\geq 1}$ are independent. For each
$n\geq 1$, set $Z_n=Y_1+\cdots +Y_n$. Then the random measure
\begin{equation*}
N(dt,dz)=\sum_{n\geq 1} \delta_{(T_n,Z_n)}(dt,dz)
\end{equation*}
is a Poisson measure on $[0,1]\times [0,\infty)$ with intensity measure
$dtdz$. We denote by 
$\cF_t:=\sigma(\{N(A), A\in \cB([0,t] \times [0,\infty))\})$
the associated filtration.

The law of large numbers ensures us that a.s., $\lim_n Z_n/n=1$.
\end{remark}
%%%%%%%%%%%%%%%%%%%%%%%%%%%%%%%%%%%%%

This leads us to the following representation of the process $M$.

%%%%%%%%%%%%%%%%%%%%%%%%%%%%%%%%%%%%%
\begin{pro}\label{prosde}
Let $N$ be as in Remark \ref{spoisson}. 
Then there exists an unique c\`adl\`ag $(\cF_t)_{t\in [0,1]}$-adapted
process $M=(M_t)_{t\in [0,1]}$ solution to
\begin{equation}\label{sde}
M_t = \intot \int_0^\infty  G(\gamma(M_\sm), z) N(ds,dz).
\end{equation}
Furthermore, $M$ is a strong Markov process with generator $\cL$ 
(see (\ref{giok}) or (\ref{gifinal})), and is a.s. increasing.
Finally, $\cJ=\{t\in [0,1]: \; \Delta M_t \ne 0\}=\{T_n:\; n\geq 1\}$.
\end{pro}
%%%%%%%%%%%%%%%%%%%%%%%%%%%%%%%%%%%%%

Proposition \ref{exi} is a consequence of this result.

\sk

%%%%%%%%%%%%%%%%%%%%%%%%%%%%%%%%%%%%%
\begin{proof}
Owing to classical arguments (see e.g. Ikeda-Watanabe \cite{iw}),
the (pathwise) existence and uniqueness of the solution to (\ref{sde}),
follow from the two following properties,
which are easily checked under $(\cH)$:

\begin{itemize}
\item
for all $y\in [0,\infty)$, $z\in [0,\infty)$,
$G(\gamma(y),z) \leq (1+z)^{-1/(1-\e)} \in L^1([0,\infty),dz)$,

\item for all $x,y \in [0,\infty)$, $\int_0^\infty |G(\gamma(x),z)
-G(\gamma(y),z)| dz =|\frac{\gamma(x)}{1-\gamma(x)}-\frac{\gamma(y)}
{1-\gamma(y)} | 
\leq C |x-y|$ (here we use that $G(\beta,z)$ is nondecreasing as a function of
$\beta$).

\end{itemize}

Indeed, it suffices to use the Gronwall Lemma and a Picard
iteration (for the norm $||X-Y||=\mathbb{E}[\sup_{[0,1]}|X_t-Y_t|]$).
The strong Markov property follows from the pathwise uniqueness,
and the monotonicity of $M$ is obvious since $G$ is non-negative.
Finally, for $\phi:[0,\infty)\mapsto\rr$ Lipschitz-continuous, we have
\begin{eqnarray*} 
\phi(M_t)&=& \phi(0) + \sum_{s\leq t} [\phi(M_\sm+\Delta M_s) - \phi(M_\sm)]\\
&=& \phi(0) + \intot \int_0^\infty  [\phi(M_\sm+G(\gamma(M_\sm), z))-
\phi(M_\sm)] 
N(ds,dz).\nonumber
\end{eqnarray*}
Taking expectations and using (\ref{gifinal}), 
we get $\mathbb{E}[\phi(M_t)]=\phi(0)+\int_0^t \mathbb{E}[\cL\phi(M_s)]ds$,
so that the generator of $M$ is indeed $\cL$.
\end{proof}
%%%%%%%%%%%%%%%%%%%%%%%%%%%%%%%%%%%%%

We also introduce a one-parameter family of L\'evy processes, and
check some regularity comparisons with $M$.

%%%%%%%%%%%%%%%%%%%%%%%%%%%%%%%%%%%%%
\begin{pro}\label{comp}
Consider the Poisson measure $N$ and the process $M$ 
introduced in Remark \ref{spoisson} 
and Proposition \ref{prosde}. For each fixed 
$\alpha\in (0,1)$, we define
\begin{equation}\label{Levybeta}
X^\alpha_t = \intot \int_0^\infty G(\alpha,z)N(ds,dz).
\end{equation}
Then $(X^\alpha_t)_{t\in[0,1]}$ is a pure-jump $(\cF_t)_{t\in[0,1]}$-adapted 
L\'evy process. Its
L\'evy measure is ${\nu }^\alpha(du)=\alpha
u^{-1-\alpha} \indiq_{[0,1]}(u)du$, for which $\beta_{{\nu}^\alpha}=\alpha$.
Almost surely,

\begin{enumerate}
\item[(i)]
for all $0<\alpha<\alpha'<1$, for all $0\leq s \leq t \leq 1$,
\begin{equation*}
0\leq X^\alpha_t-X^\alpha_s \leq X^{\alpha'}_t-X^{\alpha'}_s;
\end{equation*}

\item[(ii)] for all $0\leq s \leq t\leq 1$, 
\begin{equation*}
0\leq X^{\gamma(M_s)}_t -X^{\gamma(M_s)}_s \leq 
M_t - M_s \leq  X^{\gamma(M_\tm)}_t -X^{\gamma(M_\tm)}_s.
\end{equation*}
\end{enumerate}
\end{pro}
%%%%%%%%%%%%%%%%%%%%%%%%%%%%%%%%%%%%%

%%%%%%%%%%%%%%%%%%%%%%%%%%%%%%%%%%%%%
\begin{proof}
For each $\alpha\in (0,1)$, $X^\alpha$ is classically 
a L\'evy process, and its L\'evy measure is the image measure of $dz$
by $G(\alpha,.)$, which is nothing but $\nu^\alpha$.
Next, point (i) is not hard since for $0<\alpha<\alpha'$, we have 
$G(\alpha,z)<G(\alpha',z)$ for all $z\in (0,\infty)$.
Point (ii) is also immediate: since $u \mapsto \gamma(M_\um)$
is nondecreasing, and since $G(.,z)$ is nondecreasing for all 
$z\in(0,\infty)$,
we have a.s., for $0\leq s \leq t \leq 1$,
\begin{eqnarray*}
M_t-M_s&=&\int_{(s,t]} G(\gamma(M_\um),z) N(du,dz) \ala
&\leq&  
\int_{(s,t]} G(\gamma(M_{t-}),z) N(du,dz) = X^{\gamma(M_{\tm})}_t -
X^{\gamma(M_{\tm})}_s.
\end{eqnarray*}
Similarly, we obtain
$M_t-M_s \geq X^{\gamma(M_{s})}_t -
X^{\gamma(M_{s})}_s$, which ends the proof.
\end{proof}
%%%%%%%%%%%%%%%%%%%%%%%%%%%%%%%%%%%%%

%%%%%%%%%%%%%%%%%%%%%%%%%%%%%%%%
%%%%%%%%%%%%%%%%%%%%%%%%%%%%%%%%
%%%%%%%%%%%%%%%%%%%%%%%%%%%%%%%%
%%%%%%%%%%%%%%%%%%%%%%%%%%%%%%%%
%%%%%%%%%%%%%%%%%%%%%%%%%%%%%%%%
%%%%%%%%%%%%%%%%%%%%%%%%%%%%%%%%
\section{Local regularity}\label{locreg}

We consider a Poisson measure $N$ as in Remark \ref{spoisson}, and the associated
processes $M$, $X^\alpha$ as in Propositions \ref{prosde} and \ref{comp}. 
We start with a simple observation.

%%%%%%%%%%%%%%%%%%%%%%%%
\begin{lemma} 
Almost surely, for all $\alpha\in (0,1)$,
\begin{equation}\label{ju}
\cJ = \{t\in[0,1]: \Delta M_t \ne 0 \} =  
\{t\in[0,1]  :  \Delta X^\alpha_t \ne 0 \} = \bigcup_{n\geq 1} \{T_n\},
\end{equation}
and for all $n\geq 1$, all $\alpha\in (0,1)$,
\begin{equation*}%\label{o2}
(\Delta M_{T_n})^{\gamma(M_{T_n-})} =(\Delta X^\alpha_{T_n})^{\alpha}
=(1+Z_n)^{-1}.
\end{equation*}
\end{lemma} 
%%%%%%%%%%%%%%%%%%%%%%%%
%%%%%%%%%%%%%%%%%%%%%%%%
\begin{proof}
First, (\ref{ju}) follows from
(\ref{sde}) and (\ref{Levybeta}). Next, for $n\geq 1$,
$\Delta M_{T_n}= G(\gamma(M_{T_n-}),Z_n)=(1+Z_n)^{-1/\gamma(M_{T_n-})}$
and $\Delta X^\alpha_{T_n}= G(\alpha,Z_n)=(1+Z_n)^{-1/\alpha}$.
\end{proof}
%%%%%%%%%%%%%%%%%%%%%%%%

Next, we introduce some (random) sets of times, more or less 
well-approximated by the times of jumps of our process $M$.
The main idea is that at times well-approximated by the jump times of $M$,
the pointwise regularity of $M$ will be poor, while at times which are far from
the jump times of $M$,   $M$ will  have greater pointwise exponents.

We thus set,
for all $\delta \geq 1$,
\begin{eqnarray}\label{defad}
A_\delta &=& \bigcap_{p\geq 1} \  \bigcup_{n\geq p}  \ 
B(T_n, (\Delta M_{T_n})^{\delta \gamma(M_{T_n-})})
=\bigcap_{p\geq 1}  \ \bigcup_{n\geq p} \  
B(T_n, (\Delta X^\alpha_{T_n})^{\delta \alpha}) \\\label{adelta}
&=& \bigcap_{p\geq 1} \  \bigcup_{n\geq p}  \ B(T_n, (1+Z_n)^{-\delta}).
\end{eqnarray}
Here, $B(t,r)=(t-r,t+r)$.
Let us observe at once that
\begin{equation*}
\hbox{ for all }
\delta_1  \leq \de_2, \quad A_{\de_2} \subset A_{\delta_1}.
\end{equation*}

%Let us recall briefly the definition of Hausdorff dimensions,
%in order to fix the notation. 
%Hereafter,  $|I|$ stands for the diameter  of the set $I$. 
%For any $A \subset \rr$ and $s>0$,  the $s$-Hausdorff measure of $A$ is
%\begin{equation}\label{hs}
%\cH_s(A):= \lim_{\e \to 0} \left\{ \inf_{\cI_\e(A)} \sum_{j \geq 1} 
%|I_j|^s \right\} \in [0,\infty],
%\end{equation}
%where $\cI_\e(A)$ consists of all the coverings $(I_j)_{j\geq 1}$ of $A$ 
%composed by sets of diameter smaller than $\e$.
%The mapping  $s \mapsto \cH_s(A)$ is decreasing, and one may
%show that there exists $s_0$ (depending on A) such that for all $s>s_0$,
%$\cH_s(A)= 0$ while for all $s<s_0$, $\cH_s(A)= \infty$. 
%The Hausdorff dimension of $A$, is defined by $\dim_H(A)=s_0$.  By convention, 
%$\dim_H (\emptyset )= -\infty$. We refer
%to Falconer \cite{f} for further information.

%\sk

%The main result of this section, key for the rest of the paper, 
%is the following.

%%%%%%%%%%%%%%%%%%%%%%%%%%%%%%%%%%%%%
\begin{pro}\label{dhadelta}
Almost surely, $A_1 \supset [0,1]$. 
%and
%for all $\delta \geq 1$, $\dim_H(A_\delta) = 1/\delta$. 
\end{pro}
%%%%%%%%%%%%%%%%%%%%%%%%%%%%%%%%%%%%%

%The result is not far from being classical, and can 
%be viewed as a consequence of  \cite{j}. We write here a short   
%proof adapted to our context, in order to later compare it to 
%what happens with our Markov process. 
%We will use the following ubiquity result of \cite{DODSON} and \cite{j2} 
%(we recall here the part we need). REFERENCES PRECISES

%%%%%%%%%%%%%%%%%%%%%%%%%%%%%%%%%%%%%
%\begin{pro}
%\label{ubi1}
%Let $(t_n)_{n\in \N}$ be a sequence in $\zu$ 
%and $(l_n)_{n\in\N}$ a positive sequence decreasing 
%to $0$. For every $\delta\geq 1$, consider the set
%$$B_\delta =  \bigcap_{p\in \N} \ \bigcup_{n\geq p} \ B(t_n,(l_n)^\delta).$$
%If the Lebesgue measure of $B_1$ equals 1, then for every $\delta \geq 1$,
%there exists a Borel probability measure $m_\delta$ 
%such that $m_\delta$ is supported by $B_\delta$ (i.e. 
%$m_\delta(B_\delta)=1$) and $m_\delta(U)=0$ for every set 
%$U\subset \zu$ such that $\dim_H(U)< 1/\delta$.
%In particular, for every $\delta\geq 1$, $\dim_H (B_\delta) \geq 1/\delta$.
%\end{pro}

\begin{proof}  
Observe that $\sum_{n\geq 1} \delta_{(T_n,(1+Z_n)^{-1})}$ 
is a Poisson measure on $[0,1]\times (0,1)$ with intensity measure
$dt \mu(du)$ where $\mu(du)=du/u^2$ (because $\mu(du)$ is the image measure 
of $\indiq_{\{z>0\}}dz$ by the application $(1+z)^{-1}$).
Due to Shepp's Theorem \cite{shepp} (we use here the version used in the papers of
 Bertoin \cite{b} and Jaffard \cite[Lemma 3]{j}), it suffices to prove that
\begin{equation*}
S=\int_0^1 \exp\left(2\int_t^1 \mu((u,1))du \right) dt = \infty.
\end{equation*}
But $\mu((u,1))=u^{-1}- 1$, so that $S=\int_0^1 e^{2(t-1-\log t)}dt
\geq e^{-2} \int_0^1 dt/t^2 = \infty$.
\end{proof}

In order to characterize the pointwise exponent of $M$ at every time $t$, we need to introduce the notion of approximation rate by a Poisson process.
%%%%%%%%%%%%%%%%%%%%%%%%%%%%%%%%%%%%%

\begin{defi} \label{defdeltat}
Recall that a.s., $\delta \mapsto A_\delta$ is non-increasing and $A_1=[0,1]$.
We introduce, for any $t\in [0,1]$, 
the (random) index of approximation of $t$
\begin{equation}\label{eqdeltat}
\delta_t := \sup\{ \delta \geq 1: \; t \in A_\delta\}\geq 1.
\end{equation}
\end{defi}
%%%%%%%%%%%%%%%%%%%%%%%%%%%%%%%%%%%%%

We now are able to give the value of $h_M(t)$.

%%%%%%%%%%%%%%%%%%%%%%%%%%%%%%%%%%%%%
\begin{theo}\label{deltat} 
Almost surely,  for all $t\in [0,1]\backslash \cJ$, 
\begin{equation}
\label{egalite}
\displaystyle h_M(t)= \frac{1}{\delta_t\cdot \gamma(M_t)}.
\end{equation}

In particular, this implies that  for every   $t\in [0,1]$, $h_M(t) \leq  1/\gamma(M_t) $.
\end{theo}
%%%%%%%%%%%%%%%%%%%%%%%%%%%%%%%%%%%%%

We introduce, for $f:\rr_+\mapsto \rr$ a locally bounded function and
$t_0 \in \rr_+$, 
$$
\tilde h_f(t_0):= \sup\{\alpha>0: \exists C, \; |f(t)-f(t_0)|
\leq C |t-t_0|^\alpha\; \hbox{ in a neighborhood of }t_0\}.
$$
This notion of H\"older exponent of $f$ at $t_0$ is slightly
different of that introduced in Definition \ref{defpoint} (which may involve a polynomial).
Note that we always have $\tilde h_f(t_0)\leq h_f(t_0)$.
In the case where $f$ is purely discontinuous and increasing,
one might expect that $h_f(t_0)=\tilde h_f(t_0)$ in many cases.
This is   the case when $f$ is a L\'evy
subordinator without drift. Indeed,   
 from \cite[Lemma 4 and Proposition 2]{j}, we have a.s.
$$\mbox{for all $t\in [0,1]\backslash \cJ$,  \  for all $\alpha\in(0,1)$, \ 
$\tilde h_{X^\alpha}(t)= h_{X^\alpha}(t)=(\delta_t\cdot \alpha)^{-1}$ }$$
(here we use that $X^\alpha$ is a pure jump L\'evy process without drift
with L\'evy measure $\nu^\alpha$ satisfying $\beta_{\nu^\alpha}=\alpha$).

\medskip

%%%%%%%%%%%%%%%%%%%%%%%%%%%%%%%%%%%%%
\begin{preuve} {\it of Theorem \ref{deltat}: lowerbound.}
Let $t \in  [0,1]\backslash \cJ$ and $\e>0$ small enough be fixed.
By construction,  $M$ is continuous at $t$. 
Since $\gamma$ is also continuous,  there exists $\eta>0$ such that for all 
$s \in (t-\eta,t+\eta)$, $\gamma(M_s) \in (\gamma(M_t)-\e,\gamma(M_t)+\e)$.
We deduce from Proposition \ref{comp}-(ii) that for all
$s\in (t-\eta,t)$,
$$0 \leq M_t-M_s \leq X_t^{\alpha_\e^+} 
-X_s^{\alpha_\e^+},
$$
where $\alpha_\e^+:=\gamma(M_t)+\e$. Similarly, when $s\in (t,t+\eta)$,
$$0\leq M_s-M_t \leq X_s^{\alpha_\e^+} 
-X_t^{\alpha_\e^+}.
$$
Thus applying the definition of $\tilde h$, we conclude that
$h_M(t)\geq \tilde h_M(t)\geq \tilde h_{X^{\alpha_\e^+}}(t)=
(\delta_t.\alpha_\e^+)^{-1}$. Letting $\e$ go to zero, we deduce that
$h_M(t) \geq (\delta_t.\gamma(M_t))^{-1}$.
\end{preuve}

\sk

To prove an upper bound for $h_M(t)$, we use 
 the following result  of  \cite[Lemma 1]{JAFFOLD}.

\begin{lemma} 
\label{lemold}
Let $f:\rr\mapsto \rr$ 
be a function discontinuous on a dense set of points, and let 
$t\in \R$. Let $(t_n)_{n\geq 1}$  be a real sequence converging to $t$ such 
that at each $t_n$, $f$ has right and left limits at $t_n$ and
$|f(t_n+)-f(t_n-)|=s_n$. Then
$$
h_f(t) \leq \liminf_{n\to +\infty} \frac{\log s_n}{\log |t_n-t|}.
$$
\end{lemma}

\begin{preuve} {\it of Theorem \ref{deltat}: upperbound.}
By \eqref{defad} and \eqref{eqdeltat}, for every $\ep>0$,  
$t\in A_{\delta_t-\ep}$, so that
there exists a sequence of jump instants $(T_{n_k})_{k\geq 1}$ converging to $t$ 
such that $|t-T_\nk|\leq (\Delta M_{T_\nk})^{(\de_t-\ep)\gamma(M_{T_\nk-})}$. 
Hence, by Lemma \ref{lemold}, we get
\begin{eqnarray*} h_M(t) & \leq &  \liminf_{k\to +\infty} \frac{\log \Delta M_{T_\nk}}{\log |T_\nk-t|} \leq  \liminf_{k\to +\infty} \frac{\log \Delta M_{T_\nk}}{\log (\Delta M_{T_\nk})^{(\de_t-\ep)\gamma(M_{T_\nk-})}} \\
& \leq & \liminf_{k\to +\infty} \frac{1}{ {(\de_t-\ep)\gamma(M_{T_\nk-})}} = \frac{1}{ {(\de_t-\ep)\gamma(M_{t})}}.
\end{eqnarray*}
The last point comes from the fact that $M$ is continuous at $t$. Letting $\ep$ tend to zero yields that $h_M(t) \leq (\delta_t.\gamma(M_t))^{-1}$.

To finish the proof of Theorem \ref{deltat}, it is clear 
from \eqref{egalite} and (\ref{eqdeltat}) that
for all $t\in [0,1]\backslash \cJ$, $h_M(t) \leq 1/\gamma(M_t)$.  Finally,  $h_M(t)=0$ for all $t\in \cJ$, since $M$ jumps at $t$.
\end{preuve}

%%%%%%%%%%%%%%%%%%%%%%%%%%%%%%%%%%%%%
%%%%%%%%%%%%%%%%%%%%%%%%%%%%%%%%%%%%%
%%%%%%%%%%%%%%%%%%%%%%%%%%%%%%%%%%%%%
%%%%%%%%%%%%%%%%%%%%%%%%%%%%%%%%%%%%%
%%%%%%%%%%%%%%%%%%%%%%%%%%%%%%%%%%%%%
%%%%%%%%%%%%%%%%%%%%%%%%%%%%%%%%%%%%%

\section{Computation of the spectrum: a localized ubiquity theorem}\label{low}

We are now in a position to prove Theorem \ref{mr1}.
Recall that the Poisson measure $N=\sum_{n}\delta_{(T_n,Z_n)}$ has been introduced
in Remark \ref{spoisson}, that the process $M$ has been built in Proposition
\ref{prosde}, and that $\delta_t$ has been introduced in Definition
\ref{defdeltat}.

We will use here the 
{\em localized Jarnik-Besicovitch theorem}  of \cite{BS3}, that we explain now.
We introduced in \eqref{defad} and \eqref{eqdeltat} the approximation rate of any real number $t\in [0,1]$  by our Poisson point process. In fact, such an approximation rate can be defined for any system of points.

%%%%%%%%%%%%%%%%%%%%%%%%%%%%%%%%%%%%%
\begin{defi}
(i) A {\em system of points} $\mathcal{S}=\{(t_n,l_n)\}_{n\geq 1}$ is a 
$\zu\times (0,\infty)$-valued sequence such that $l_n$ decreases to $0$
when $n$ tends to infinity.

\noindent (ii)  $\mathcal{S}=\{(t_n,l_n)\}_{n\geq 1}$ is said to satisfy
the covering property if 
\begin{equation}\label{cp}
\bigcap_{p\geq 1}  \ \bigcup_{n\geq p} \  B(t_n,l_n)
\supset [0,1].
\end{equation}

\noindent (ii)
For $t\in \zu$, the approximation rate of $t$ by $\mathcal{S}$ is defined as
\begin{equation}
\label{defdelta2}
\delta_t= \sup\{\delta\geq 1: t \mbox{ belongs to an infinite number of balls  } B(t_n,l_n^\delta)\}.
\end{equation}
\end{defi}
%%%%%%%%%%%%%%%%%%%%%%%%%%%%%%%%%%%%%

Set  $\lambda_n:=(1+Z_n)^{-1}$. Then
$\mathcal{P}= \{(T_n,\lambda_n)\}_{n\geq 1}$ is a system of points. 
Of course formula \eqref{defdelta2} coincides with formula 
\eqref{eqdeltat} when the system of points is $\mathcal{P}$. 
This system $\mathcal{P}$ is a Poisson point process with intensity measure
\begin{equation}
\label{defpoiss}
\Lambda(s,\lambda)={\bf 1}_{\zu\times (0,1)}
(s,\lambda)ds \frac{d\lambda}{\lambda^2}.
\end{equation}

Let us state the result of \cite[Theorem 1.3]{BS3}. 
The definitions of a {\em weakly redundant system} and the 
{\em condition  {\bf (C)}} are recalled in next section. 
There, the  Poisson system $\mathcal{P}$ is shown to enjoy these properties 
almost surely.

%%%%%%%%%%%%%%%%%%%%%%%%%%%%%%%%%%
\begin{theo}
\label{main}
Consider a weakly redundant system $\cS$ satisfying the covering property
(\ref{cp}) and condition {\bf (C)}.
Let  $I=(a,b)\subset\zu$ and $f:I \to [1,+\infty)$ be   continuous
at every $t\in I\backslash \cZ$, for some countable
$\cZ \subset [0,1]$.
Consider
\begin{equation*}
S(I,f)  =   \left\{t\in I: \ \delta_t \geq f(t) \right\}
\;\;\hbox{ and }\;\;
\widetilde S(I,f)  =  \left\{t\in I: \ \delta_t  =  f(t) \right\}.
\end{equation*}
Then
\begin{equation*}
\dim_H S(I,f)= \dim_H\widetilde S(I,f)  =  \sup \{1/f(t): t\in I\backslash 
\cZ\}.
\end{equation*}
\end{theo}
%%%%%%%%%%%%%%%%%%%%%%%%%%%%%%%%%%

Observe that $\cP$ satisfies the covering property due to Proposition
\ref{dhadelta}.
We assume for a while that a.s., 
the Poisson system $\mathcal{P}$ is weakly redundant and fulfills {\bf (C)}.

\begin{pro}\label{specbiz} 
Consider the process $M$ built in Proposition \ref{prosde}.
Almost surely, 
\begin{enumerate}
\item[(i)] for all $I=(a,b) \subset [0,1]$, all $\kappa \in (0,1)$,
$$
\dim_H \{ t\in I: \; h_M(t)=\kappa/\gamma(M_t)\}  = \dim_H \{ t\in I: \; h_M(t)\leq \kappa/\gamma(M_t)\}= \kappa;
$$
\item[(ii)]
for Lebesgue-almost every $t\in [0,1]$, $h_M(t)=1/\gamma(M_t)$.
\end{enumerate}
\end{pro}

\begin{proof} By Theorem \ref{deltat} and since $\cJ$ is countable,
we observe that for $I=(a,b)\subset[0,1]$,
$$
\dim_H \{ t\in I: \; h_M(t)=\kappa/\gamma(M_t)\}  = 
\dim_H \{ t\in I: \; \delta_t=1/\kappa\}.
$$

Let $\kappa \in (0,1)$.
Since the Poisson system $\cP=\{(T_n,\lambda_n)\}_{n\geq 1}$ 
satisfies all the required conditions, we may apply Theorem
\ref{main} with the constant function $f\equiv 1/\kappa$
and get $\dim_H \{ t\in I: \; h_M(t)=\kappa/\gamma(M_t)\}=\kappa$.
The same arguments hold for 
$ \dim_H \{ t\in I: \; h_M(t)\leq \kappa/\gamma(M_t)\}$, which concludes the
proof of (i).
Next, we write, for $I=(a,b)\subset[0,1]$,
\begin{equation}
\label{decomp}
I = \Big \{t\in I: h_M(t)=1/\gamma(M(t))\Big\} \ \bigcup 
\left( \cup_{n\geq 1} S_n \right),
\end{equation}
where $S_n:= \Big\{t\in I: h_M(t)\leq (1-1/n)/\gamma(M(t)) \Big\}$. 
By (i), for every $n\geq 1$, the Lebesgue measure of the set $S_n$ is zero since it is of Hausdorff dimension strictly less than $1$. 
We deduce from \eqref{decomp} that for Lebesgue-a.e. $t\in I$,
$h_M(t)=1/\gamma(M(t))$. Since this holds for any $I=(a,b)\subset [0,1]$,
the conclusion follows.
\end{proof}

\begin{preuve} {\it of Theorem \ref{mr1}.} By Remark \ref{enough},
it suffices to prove \eqref{mfs1.5}.
Let $I=(a,b)\subset[0,1]$. By Theorem \ref{deltat}, for all $h\geq 0$,
$$
D_M(I,h)=\dim_H \{t\in (a,b): \; h_M(t)=h\}=\dim_H \{t\in (a,b): \; \delta_t=
(h\cdot \gamma(M_t))^{-1} \}.
$$
The jump times   $\cJ$ are countable and of exponents zero for $M$, so they  do not interfere in the computation of Hausdorff dimensions.

\smallskip

For $s\in\cJ\cap (a,b)$ a jump time of $M$ and $h\in (0,1/\gamma(M_\sm))$,
consider the function $f_s$ defined on the interval $I_s= (a,s)$ by
$f_s(t)=(h . \gamma(M_t))^{-1}$. This function is continuous on 
$I_s\backslash\cJ$, and satisfies, for every $t\in I_s$, 
$f_s(t) \geq  (h \cdot \gamma(M_{\sm}))^{-1}\geq 1.$
Applying Theorem \ref{main} to  the Poisson system $\cP=\{(T_n, \lambda_n)\}_{n\geq 1}$ 
(which satisfies all the required conditions), we obtain
$$ 
\dim_H  \{t\in I_s: \ \delta_t  =  (h  \cdot  \gamma(M_t))^{-1} \}  =  
\sup \{h \cdot \gamma(M_t): t\in I_s\backslash \cJ\}  = h\cdot \gamma(M_\sm),
$$
since $\gamma(M_t)$ increases to $\gamma(M_\sm)$ as $t$ increases to $s$.
Hence, for every $s\in \cJ \cap (a,b)$, for every $h$ such that $ 0 <h \leq 1/\gamma(M_{\sm})$, we have
$$  
\dim_H  \Big\{E_M(h) \cap I_s \Big\}  =   h\cdot \gamma(M_\sm).
$$
Furthermore, for  $s\in \cJ \cap (a,b)$, when $h \geq 1/\gamma(M_{\sm})$, 
we have
$E_M(h) \cap [s,b]=\emptyset$. Indeed, by Theorem \ref{deltat},
for $t\geq s$, $h_M(t)\leq 1/\gamma(M_t)\leq  1/\gamma(M_s)<1/\gamma(M_\sm)$.

Let now $h\geq 0$ be fixed. Then using the density of $\cJ$,
\begin{eqnarray*}
\displaystyle
E_M(h)\cap I &=& \left(\bigcup_{s\in \cJ\cap (a,b), \gamma(M_\sm)<1/h} (E_M(h)
\cap (a,s))
\right) \\
&& \bigcup 
\left(\bigcup_{s\in \cJ\cap (a,b), \gamma(M_\sm)\geq 1/h} (E_M(h)\cap [s,b))
\right)
\end{eqnarray*}
As noted previously, the second term of the right hand side is empty.
Thus we get, since $\cJ$ is countable,
\begin{eqnarray*}
D_M(I,h) &=& \sup \{ \dim_H  (E_M(h)\cap I_s) : \; s\in \cJ \cap I,
\; \gamma(M_\sm)<1/h\}\\
&=& \sup \{h\cdot \gamma(M_\sm) : s\in \cJ\cap I 
\mbox { and } h\cdot \gamma(M_\sm) <1\},
\end{eqnarray*}
which was our aim. Observe that if $h\geq 1/\gamma(M_a)$, this 
formula gives $D_M(I,h)=-\infty$.
\end{preuve}

%%%%%%%%%%%%%%%%%%%%%%%%%%
%%%%%%%%%%%%%%%%%%%%%%%%%%
%%%%%%%%%%%%%%%%%%%%%%%%%%
%%%%%%%%%%%%%%%%%%%%%%%%%%
%%%%%%%%%%%%%%%%%%%%%%%%%%
\section{Study of the distribution of the Poisson point process}
\label{properties}

To conclude the proof, we only have to check that $\mathcal{P}$ is  
a weakly redundant system satisfying {\bf (C)}. Recall that
$\mathcal{P}=\{(T_n,\lambda_n)\}_{n\geq 1}$
is a Poisson point process with intensity measure \eqref{defpoiss}.

%%%%%%%%%%%%%%%%%%%%%%%%%%
%%%%%%%%%%%%%%%%%%%%%%%%%%
%%%%%%%%%%%%%%%%%%%%%%%%%%
%%%%%%%%%%%%%%%%%%%%%%%%%%
%%%%%%%%%%%%%%%%%%%%%%%%%%
\subsection{Weak redundancy and condition {\bf (C)}}

The weak redundancy property asserts that the balls $B(t_n,l_n)$ naturally associated with a system of points $\mathcal{S}$ do not overlap excessively. The precise definition is the following.

%%%%%%%%%%%%%%%%%%%%%%%%%%
\begin{defi}
\label{defweak} Let $\cS=\big \{(t_n,l_n)\big \}_{n\ge 1}$ be a 
system of points. For any integer $j\ge 0$ we set
\begin{equation*}
%\label{deftj}
\mathcal{T}_{j} =\big \{n: \ \ 2^{-(j+1)}<l_n\le
2^{-j}\big \}.
\end{equation*}
\noindent

Then $\cS$ is said to be {\em weakly redundant} if $t_n\neq t_{n'}$ for all
$n\neq n'$ and if there exists  a non-decreasing sequence of positive
integers $(N_{j})_{j\ge 0}$ such that

\noindent
(i) $\lim_{j\to\infty}({\log_2 N_{j}})/{j}=0$.

\noindent
(ii) for every $j\ge 1$, $\mathcal{T}_j$ can be decomposed into
$N_{j}$ pairwise disjoint subsets (denoted
$\mathcal{T}_{j,1},\dots,\mathcal{T}_{j,N_{j}} $) such that for each
$1\le i\le N_{j} $, the balls $B(t_n,l_n)$, $ n\in \mathcal{T}_{j,i}$, are pairwise disjoint.
\end{defi}
%%%%%%%%%%%%%%%%%%%%%%%%%%%%%%%%%%
In other words, for every  $x\in \zu$, $x$ cannot belong to more than $N_j$ balls $B(t_n,l_n)$ with $2^{-j-1}<l_n\leq 2^j$. 

\sk

As shown in \cite{BS4}, Proposition 6.2, $\cP$ is weakly redundant.

\sk

A weak redundant system  do not necessarily  satisfy the conclusion
of Theorem \ref{main}. This is the reason why condition {\bf{(C)}}, which imposes finer properties  on the distribution of the system $\mathcal{S}$,  has to be introduced.

\smallskip

We denote by $\Phi$ the set of   functions $\varphi:\mathbb{R}_+ \to \R_+$  
such that
\begin{itemize}
\item
$\varphi$ is increasing, continuous and $\varphi(0)=0$,
\item
$r^{-\varphi(r)}$ is decreasing and tends to $\infty$ as $r$ tends to 0, 
\item
$r^{\alpha-\beta \varphi(r)}$ is increasing near 0 for all $\alpha,\beta>0$. 
\end{itemize}
For example, the function $\varphi(r)= \frac{\log|{\log(r)}|}{|\log r|}$, 
defined for $r\geq 0$ close enough to 0, has the required behavior around 0.

%%%%%%%%%%%%%%%%%%%%%%%%%%%%%%%%%%

\begin{defi}
\label{def3}
Suppose that  a system of points  $\cS=\{(t_n,l_n)\}_{n\geq 1}$ is weakly 
redundant and adopt the notation of Definition \ref{defweak}.
For every  $\varphi\in\Phi$  and for any $j\ge 1$, we define
\begin{eqnarray*}
\psi( j,\ph) & = & \max \left\{m\in\mathbb{N}:   N_m\cdot 2^{m}\le 2^{j(1-\varphi(2^{-j}))}\right\}.\end{eqnarray*}
\end{defi}
%%%%%%%%%%%%%%%%%%%%%%%%%%%%%%%%%%
Of course, the   sequence of integers   $(  N_j)_{j\ge 1}$ is the one appearing in the  Definition \ref{def3} of the weak redundancy.

Obviously $\psi( j,\ph)\leq j$, for every $\ph$, $(N_j)$ and $j$.

\sk

For any dyadic interval $U=[k2^{-j}, (k+1)2^{-j})$, we set $g(U)=j$, i.e. $g(U) $ is  the dyadic generation of $U$. We denote by $\cG_j$ the set of all dyadic
intervals of generation $j$. Finally, we denote by $\cG_*:=\bigcup_{j\geq 0}
\cG_j$ the set of all dyadic intervals.

%%%%%%%%%%%%%%%%%%%%%%%%%%%%%%%%%%
\begin{defi}\label{def2}
Suppose that  a system of points  $\cS=\{(t_n,l_n)\}_{n\geq 1}$ is weakly 
redundant and adopt the notation of Definitions \ref{defweak} and \ref{def3}.
Let $\varphi \in \Phi$. For every dyadic interval $V\in \cG_j$ and every $\delta>1$,   the property $\mathcal{P}(V,\delta)$ is satisfied when  there exists an integer $n\in \mathcal{T}_{j}$ such that $t_n\in V$ and 
$$
B\left(t_n,(l_n)^{\delta} \right) \  \bigcap  \ \left\{   t_p : 
\begin{array}{l}  \ \  p\in  \mathcal{T}_k, \ \ \mbox{ where $k$ satisfies  }  \ \\ \ \  \psi( j,\ph) \, \le \, k \, \le\,  - \log_2   (l_n)^{\delta}  +4 
\end{array} \right\} =\{t_n\}.
$$
\end{defi}
%%%%%%%%%%%%%%%%%%%%%%%%%%%%%%%%%%

Let us try to give an   intuition of the meaning of  $\mathcal{P}(V,\delta )$ for a dyadic interval $V$.
$\mathcal{P}(V,\delta )$  holds true when, except $t_n$,   the family of points $(t_p)_{p\geq 1}$ "avoids" the ball $B\left(t_n,(l_n)^{\delta } \right) $  when $p$ describes  all the sets $\mathcal{T}_k$, for $k$ ranging between $g(V)$ and $ - \log_2   (l_n)^{\delta }  +4$, i.e. between the dyadic generations of $
B\left(t_n,l_n  \right) $ and $B\left(t_n,(l_n)^{\delta } \right) $.

\sk

This condition seems to be reasonable, maybe not for all   dyadic intervals, but at least for a large number among them.  {Condition    {\bf (C)}} is meant to ensure the validity of $\mathcal{P}(V,\delta )$ for a sufficient set of intervals $V$ and approximation degrees $\delta$.

\smallskip

%%%%%%%%%%%%%%%%%%%%%%%%%%%%%%%%%%

\begin{defi}
\label{def4}
Suppose that  a system of points  $\cS=\{(t_n,l_n)\}_{n\geq 1}$ is weakly 
redundant and adopt the notation of Definitions \ref{defweak} and \ref{def3}.
The system $\cS$ is said to satisfy condition {\bf (C)} if there exist:
\begin{itemize}
\item
a function $\varphi\in\Phi$, 
\item
a continuous function $\kappa:(1,+\infty) \to (0,1]$,
\item
a dense subset $\Delta$ of $(1,\infty)$,
\end{itemize}
such that
for every $\delta\in \Delta$,  for every dyadic  interval $U \in \cG_*$,
there are infinitely many integers $j\ge g(U)$ satisfying
\begin{equation*}
%\label{minqjd}
\#\mathcal{Q}(U,j,\delta) \, \ge \, \kappa(\delta )  \cdot  2^{d(j -g(U))} ,
\end{equation*}
where
$$
\mathcal{Q}(U,j,\delta)=\left\{V\in \mathcal{G}_j:\ V\subset U \ \mbox{ and } \ \ \mathcal{P}\big (V,\delta \big )\text{ holds}\right \}.
$$
\end{defi}
%%%%%%%%%%%%%%%%%%%%%%%%%%%%%%%%%%
%%%%%%%%%%%%%%%%%%%%%%%%%%%%%%%%%%

%%%%%%%%%%%%%%%%%%%%%%%%%%
%%%%%%%%%%%%%%%%%%%%%%%%%%
%%%%%%%%%%%%%%%%%%%%%%%%%%
%%%%%%%%%%%%%%%%%%%%%%%%%%
%%%%%%%%%%%%%%%%%%%%%%%%%%%%%%%%%%%%%%%%%%%%%%%%%%%%%%%%%%%%%%%
%%%%%%%%%%%%%%%%%%%%%%%%%%%%%%%%%%%%%
%%%%%%%%%%%%%%%%%%%%%%%%%%%%%%%%%%%%%
%%%%%%%%%%%%%%%%%%%%%%%%%%%%%%%%%%%%%
\subsection{Proof of {\bf (C)} for the Poisson process $\cP$.}
We only need to find  a function $\varphi\in\Phi$ and a continuous function $\kappa:(1,+\infty) \to \R_+^*$ such that for every $\delta>1$,  with probability 1, for every $U\in\mathcal{G}_*$, there are infinitely many integers $j\ge g(U)$ satisfying  $\#\mathcal{Q}(U,j,\delta)\ge \kappa(\de) \cdot 2^{j-g(U)}$. Then, for any countable and dense subset $\Delta$ of $(1,\infty)$, with probability 1, for every $\delta\in\Delta$, for every $U\in\mathcal{G}_*$, there are infinitely many integers $j\ge g(U)$ satisfying  $\#\mathcal{Q}(U,j,\delta)\ge \kappa(\de) \cdot 2^{j-g(U)}$.

In fact, any $\varphi\in\Phi$ is suitable. 
\smallskip

Let $\varphi \in \Phi$ and $\delta>1$. For $U\in\mathcal{G}_*$ and $V \subset U$ such that $V\in  \bigcup_{j >  g(U)}\mathcal{G}_j$, let us introduce the event
$$
\mathcal{A}(U,V,\delta)= \left .\begin{cases}   \exists\ n\in \mathcal{T}_{g(V)} \mbox{ such that } \ T_n\in V \mbox{ and }\\    
B(T_n,(\lambda_n)^{\delta }) \bigcap  \ \left(\bigcup_{\psi({g(V)},\ph) \le k \le h(V)} \mathcal{T}_j\right) =\{T_n\}
\end{cases} \!\!\!\!\!\right \}
$$
where $h(V)= \big [\delta (g(V)+1)  \big ]+4$.  Recall that $n\in \mathcal{T}_{g(V)} $ means that $2^{-g(V)-1} < \lambda_n \leq 2^{-g(V)}$.
By construction, we have the inclusion $\mathcal{A}(U,V,\delta) \subset \{\mathcal{P}(V,\delta )\text{ holds}\}$.
%  as soon as  $2^{-2g(V)\varphi(2^{-g(V)})}\le 1/2$.

For every $j\ge 1$, let $\widetilde {\mathcal{G}}_j=\big\{[2k\cdot 2^{-j},(2k+1)\cdot 2^{-j}]: 0\le k\le 2^{j}-1\big\}$.  The restrictions of the Poisson point process to the strips $V\times (0,1)$, where $V$ describes $\widetilde {\mathcal{G}}_j$, are independent. Consequently, the events $\mathcal{A}(U,V,\delta)$, when $V \in \widetilde {\mathcal{G}}_j$  and $V\subset U$, are independent (we must separate the intervals in $\widetilde {\mathcal{G}}_j$   because if $V\in \mathcal{G}_j$, $T_n\in V$ and $\lambda_n\le 2^{-j}$, then $B(T_n,(\lambda_n)^{\delta })$ may overlap  with the neighbors of $V$).  

We denote by $X(U,V,\delta)$ the random variable  $\mathbf{1}_{\mathcal{A}(U,V,\delta)}$. For a given generation $j>g(U)$, the random variables $(X(U,V,\delta))_{V\in \widetilde{\mathcal{G}}_j}$ are i.i.d Bernoulli variables, whose common  parameter is denoted by $p_{j}(\de)$. The following Lemma holds.
%%%%%%%%%%%%%%%%%
\begin{lemma} \label{bbb}
There exists  a continuous function $\kappa_1:(1,+\infty) \to (0,1)$   such that for every $j\geq 1$, $p_j(\de)\ge\kappa_1(\de)$.
\end{lemma} 
%%%%%%%%%%%%%%%%%

Let us assume Lemma \ref{bbb} for the moment. By definition we have
$$
\#\mathcal{Q}(U,j,\delta)\ge \sum_{V\in \widetilde{\mathcal{G}}_j:\ V\subset U} X(U,V,\delta).
$$
The right hand side of this inequality is a binomial variable of parameters $(2^{j-g(U)},p_j(\de))$, with $p_j(\de)\ge \kappa_1(\de)>0$. Consequently, there exists a constant $\kappa(\de)>0$ satisfying 
\begin{equation}
\label{eqfin}
\mathbb{P}\Big (\sum_{V\in \mathcal{G}_j,\ V\subset U} X(U,V,\delta) \ge \kappa(\de) \cdot 2^{j-g(U)}\Big )\ge 1/2
\end{equation}
provided that  $j$ is large enough. The continuity of $\kappa$ with respect to the parameter $\de>1$ follows from the continuity of $\kappa_1$.

Let $(j_n)_{n\ge 1}$ be the sequence defined inductively by $j_1=g(U)+1$ and $j_{n+1}=(j_n+1)\delta +5$. We notice that the events $E_n$ defined for $n\geq 1$ by 
$$E_n = \{\#\mathcal{Q}(U, j_n,\delta)\ge \kappa(\de) \cdot  2^{j_n-g(U)}\}$$
are independent. Moreover,   (\ref{eqfin}) implies that $\sum_{n\ge 1}\mathbb{P}(E_n)=+\infty$. The Borel-Cantelli Lemma yields that, with probability 1, there is an infinite number of generations $j_n$ satisfying  $\#\mathcal{Q}(U,  j_n,\delta)\ge \kappa(\de) \cdot  2^{j-g(U)}$. This holds true for every $U\in \mathcal{G}_*$ almost surely, hence almost surely  for every $U\in \mathcal{G}_*$. Condition {\bf (C)} is proved.

\smallskip 

We prove    Lemma~\ref{bbb}. 
For every $V\in \mathcal{G}_*$, let us introduce the sets 
$$
S_V  =   V\times [2^{-(g(V)+1)},2^{- g(V)}]\ \mbox{ and } \ 
\widetilde S_V  =  V\times [2^{-h(V)},2^{-\psi({g(V)},\ph)}].
$$
We denote by $N_V$ and $\widetilde N_V$ respectively the cardinality of $\mathcal{P}\cap S_V$ and $\mathcal{P} \cap (\widetilde S_V \setminus S_V)$. These random variables $N_V$ and $\widetilde N_V$  are independent. We set  $l_V=\Lambda(S_V)$ and $\widetilde l_V=\Lambda(\widetilde S_V)$ ($\Lambda$ is the intensity of the Poisson point process (\ref{defpoiss})). Due to the form of the intensity $\Lambda$,  $N_V$ and $\widetilde N_V$ are  Poisson random variables of parameter $l_V=1$ and $\widetilde l_V  = 2^{-g(V)}\big (2^{h(V)} -2^{g(V)+1}+2^{g(V)}-2^{\psi({g(V)},\ph)}\big )$ respectively. Observe  that $\widetilde l_V  \le  2^{h(V)-g(V)}$ since by definition $\psi({g(V)},\ph) \le g(V)$.

We also consider   two sequences of random variables in $\R^2$  $(\xi_p=(X_p,Y_p))_{p\ge 1}$ and $(\widetilde\xi_q=(\widetilde X_q,\widetilde Y_q))_{q\ge 1}$ such that 
\begin{eqnarray*}
\mathcal{P}\cap S_V & = & \{\xi_p:\ 1\le p \le N_V\}\\
\mathcal{P}\cap (\widetilde S_V\setminus S_V) & = & \{\widetilde \xi_q:\ 1\le q\le \widetilde N_V\}.
\end{eqnarray*}
The event $\mathcal{A}(U,V,\delta)$ contains the event $\widetilde {\mathcal{A}}(U,V,\delta)$ defined as 
$$
\left\{
\,  N_V=1  \mbox{ and }  
\,  B\big (X_1, Y_1^{\delta )}\big)  \, \bigcap \, \big \{ \widetilde X_q:   1\le q\le \widetilde N_V\big\}=\{X_1\}   
\right\},
$$
where  $\xi_1=(X_1, Y_1)$. The difference between $\mathcal{A}(U,V,\delta)$ and $\widetilde {\mathcal{A}}(U,V,\delta)$ is that the latter one imposes that there is one, and only one, Poisson point in $S_V$.
We have
\begin{eqnarray*}
&&\mathbb{P}(\widetilde {\mathcal{A}}(U,V,\delta))\\
& = & \mathbb{P}\Big( \Big\{B\big (X_1, Y_1^{\delta }\big)  \bigcap  \big \{ \widetilde X_q:   1\le q\le \widetilde N_V\big\} =  \emptyset \, \Big| \, \big\{N_V=1\big\}\Big\} \Big) \\
&&  \times   \ \mathbb{P}(\{N_V=1\})\\
& = & \mathbb{P}\Big(  \Big\{\forall \, 1\le q\le \widetilde N_V, \  \widetilde X_q \not\in B\big (X_1, Y_1^{\delta }\big) \} \ \Big| \ \big\{N_V=1 \big\} \Big\}\Big) \times  e^{-1}.\end{eqnarray*}
where  $\mathbb{P}(\{N_V=1\}) = e^{-1}$  since $N_V$ is a Poisson random variable of parameter 1.
The random variables $\widetilde X_q$ are independant and uniformly 
distributed in $V$. Thus,  
\begin{eqnarray*}
&&   \mathbb{P}\Big(  \Big\{\forall \, 1\le q\le \widetilde N_V, \  \widetilde X_q \not\in B\big (X_1, Y_1^{\delta  }\big) \} \ \Big| \ \big\{N_V=1 \big\} \Big\}\Big) \\
&& \ge  \mathbb{E}\Big ( \Big [1-\frac{\ell\big (B(X_1,Y_1^{\delta   })\big )}{2^{-g(V)}}\Big ]^{\widetilde N_V}\Big ).
\end{eqnarray*}
Observe that, since $\delta>1$,  provided that $g(V)$ is large  enough,  conditionally on $\{N_V\geq 1\}$,  $\ell\big(B(X_1,Y_1^{\delta })\big )\le 2^{-g(V)\delta }$.  This implies that
\begin{eqnarray}\label{aaa}
\mathbb{P}(\widetilde {\mathcal{A}}(U,V,\delta))\ge  e^{-1} \times \mathbb{E}\Big (\big [1-2^{-g(V)(\delta-1 ))}\big ]^{\widetilde N_V}\Big ).
\end{eqnarray}

Let us define $\eta_{g(V)}=2^{-g(V)(\delta-1 )}$. Using that $\widetilde N_V$ is a Poisson random variable of parameter $\widetilde l_V$, a  classical calculus shows that (\ref{aaa}) can be rewritten as 
$$
\mathbb{P}(\widetilde {\mathcal{A}}(U,V,\delta))\ge e^{-1} e^{-\widetilde l_V \cdot \eta_{g(V)}}.
$$
By using the definition  of $h(V)=\big [(g(V)+1) \delta \big ]+4$, we
can get
\begin{eqnarray*}
\widetilde l_V \cdot  \eta_{g(V)} & \leq & 2^{h(V) -g(V))}  2^{-g(V)(\delta-1 )}\le  16 \cdot 2^\delta.
\end{eqnarray*}
Thus,  $\widetilde l_V \eta_{g(V)}$ is   bounded from above independently of $V$ by a continuous function of $\de$.  Consequently, $\mathbb{P}(\widetilde {\mathcal{A}}(U,V,\delta))$, and thus $\mathbb{P}( {\mathcal{A}}(U,V,\delta))$,  is   bounded from below by some quantity $\kappa_1(\de)$ which is strictly positive and continuously dependent on $\de>1$.  Lemma \ref{bbb} is proved.

%%%%%%%%%%%%%%%%%%%%%%%%%%%%%%%%%%%%%%%%%
%%%%%%%%% Biblio %%%%%%%%%%%%%%%%%%%%%%%%
%%%%%%%%%%%%%%%%%%%%%%%%%%%%%%%%%%%%%%%%%
\section{Some tangent stable L\'evy processes}
\label{Tangent}

Let $t_0\geq 0$ be fixed. Then one observes (recall Theorems
\ref{thjaff} and \ref{mr1}) that the local multifractal spectrum $D_M(t_0,.)$
of our process $M$ essentially coincides with the multifractal spectrum of
a stable L\'evy process
with L\'evy measure $\gamma(M_{t_0}) u^{-1-\gamma(M_{t_0})} du$.
A possible explanation for this is that such a stable 
L\'evy process is {\it tangent} to our process.

\begin{pro}\label{tangent}
Let $t_0\ge 0$ be fixed. Conditionally on $\mathcal{F}_{t_0}$, 
the family of processes 
$\displaystyle \Big (\frac{M_{t_0+\alpha t}-M_{t_0}}
{\alpha^{1/\gamma(M_{t_0})}}\Big )_{t\in [0,1]}$ converges in law, 
as $\alpha\to 0^+$, to a stable L\'evy 
subordinator with L\'evy measure $\gamma(M_{t_0}) u^{-1-\gamma(M_{t_0})} du$.
Here the Skorokhod space of c\`adl\`ag functions on $[0,1]$ is endowed
with the uniform convergence topology (which is stronger than
the Skorokhod topology).
\end{pro}

One might conjecture that under many restrictive conditions, a result of the
following type might hold: if a process $(M_t)_{t\in [0,1]}$ has a 
tangent process $(Y^{t_0}_t)_{t\in [0,1]}$ 
at time $t_0$, then $D_M(t_0,.)$ coincides with the multifractal spectrum
of $Y^{t_0}$.
This would allow, 
for example, to generalize our results to the study of the multifractal
spectrum of any reasonable jumping S.D.E. (which is always tangent, in some
sense, to a L\'evy process).

\begin{proof}
Using the Markov property, it suffices to treat the case $t_0=0$.
Let thus $N(ds,dz)$ be a Poisson measure on $[0,1]\times[0,\infty)$ with
intensity measure $dsdz$. Recall that
\begin{eqnarray*}
M_t=\int_0^t\int_0^\infty (1+z)^{-1/\gamma(M_\sm)} N(ds,dz)
\end{eqnarray*}
and introduce the L\'evy processes 
\begin{eqnarray*}
&L_t=\intot\int_0^\infty (1+z)^{-1/\gamma(0)}N(ds,dz), \quad
\tilde L_t= \int_0^t\int_1^\infty z^{-1/\gamma(0)}N(ds,dz), \\
&
S_t= \intot\int_0^\infty z^{-1/\gamma(0)}N(ds,dz).
\end{eqnarray*}
One immediately checks that $(L_t)_{t\in [0,1]}$ and  $(\tilde L_t)_{t\in [0,1]}$
have the same law, and that
$(S_t)_{t\in [0,1]}$ is a stable L\'evy process with L\'evy measure 
$\gamma(0) u^{-1-\gamma(0)} du$. Thus our aim is to prove that
$(\alpha^{-1/\gamma(0)}M_{\alpha t})_{t\in[0,1]}$ tends in law to 
$(S_t)_{t\in [0,1]}$.
First, $(\alpha^{-1/\gamma(0)}S_{\alpha t})_{t\in[0,1]}$ 
has the same law as $(S_t)_{t\in[0,1]}$ for each $\alpha>0$.
Next, it is clear that
$$\mathbb{P}  \Big[(\alpha^{-1/\gamma(0)}\tilde L_{\alpha t})_{t\in [0,1]} 
= (\alpha^{-1/\gamma(0)}S_{\alpha t})_{t\in [0,1]} \Big] \geq \mathbb{P}\Big[N([0,\alpha]\times 
[0,1])=0\Big] =e^{-\alpha},$$ which
tends to $1$ as $\alpha$ tends to $0$. 

We will now show that $\alpha^{-1/\gamma(0)} \Delta_{\alpha}$
tends to $0$ in probability, where
$$
\Delta_{t}:=\sup_{[0,t]} |M_{s} - L_{s}|
=\int_0^{t}\int_0^\infty [(1+z)^{-1/\gamma(M_\sm)}-(1+z)^{-1/\gamma(0)}]N(ds,dz),
$$
and this will conclude the proof.
A first computation, using that $\gamma(y)\leq 1-\e<1$ by assumption, 
shows that for all $t\geq 0$,
$$
\E[M_t]= \int_0^t \int_0^\infty \E[ (1+z)^{-1/\gamma(M_s)}] dzds \leq
\int_0^t \int_0^\infty (1+z)^{-1/(1-\e)} dzds \leq C t
$$ 
for some constant $C$.
Next, we introduce, for $\eta>0$, the stopping time $\tau_\eta=\inf\{t\geq 0:\;
\gamma(M_t)>\gamma(0)+\eta\}$. Denoting by $A$ the Lipschitz constant of
$\gamma$, one easily gets 
$$
\mathbb{P} [\tau_\eta<\alpha] \leq \mathbb{P} [M_{\alpha} \geq \eta/A] 
\leq (A/\eta) \E[M_{\alpha}] \leq C A\alpha/\eta=C_{\eta} \alpha.
$$
Now for $\beta \in (\gamma(0)+\eta,1]$, by subadditivity, one obtains,
for all $t\geq 0$,
\begin{eqnarray*}
\E[M_{t\land \tau_\eta}^\beta]&\leq& \E\left[\int_0^{t\land \tau_\eta} \int_0^\infty 
(1+z)^{-\beta/\gamma(M_\sm)} N(ds,dz) \right]\\
&=& \E\left[\int_0^{t\land \tau_\eta} \int_0^\infty 
(1+z)^{-\beta/\gamma(M_s)} dzds \right]\\
&\leq&  
\E\left[\int_0^{t\land \tau_\eta} \int_0^\infty (1+z)^{-\beta/(\gamma(0)+\eta)} 
dzds \right] \leq C_{\eta,\beta} t.
\end{eqnarray*}
Let us introduce for $m\geq 0$ the quantity   $\kappa(m):=1/\gamma(0)-1/\gamma(m)\geq 0$. Still  by subadditivity, for $\beta \in (\gamma(0)+\eta,1]$, for $t\geq 0$, we have
\begin{eqnarray*}
\E[\Delta_{t\land \tau_\eta}^\beta ] &\leq & 
\E \left[\int_0^{t\land \tau_\eta}\int_0^\infty 
[(1+z)^{-1/\gamma(M_s)}-(1+z)^{-1/\gamma(0)}]^\beta dzds\right]\\
&\leq & \E \left[\int_0^{t\land \tau_\eta}\int_0^{2^{1/\kappa(M_s)}-1} 
(1+z)^{-\beta/\gamma(0)}  [(1+z)^{\kappa(M_s)}-1]^\beta dz
ds\right]\\
&&+ \E \left[\int_0^{t\land \tau_\eta}\int_{2^{1/\kappa(M_s)}-1}^\infty 
(1+z)^{-\beta/\gamma(M_s)} dzds\right]=:I_t^{\beta,\eta}+J_t^{\beta,\eta} .
\end{eqnarray*}
But for $z\leq 2^{1/\kappa(m)}-1$, there holds
$(1+z)^{\kappa(m)}-1 \leq C \kappa(m) \log(1+z) \leq C m \log(1+z)$, the last inequality following from the facts that  $\kappa(0)=0$ and   $\kappa$ is Lipschitz-continuous.
Hence
\begin{eqnarray*}
I_t^{\beta,\eta}&\leq & \E \left[\int_0^{t\land \tau_\eta} 
C M_s^\beta \int_0^\infty (1+z)^{-\beta/\gamma(0)}(\log(1+z))^\beta dzds \right]
\leq C_\beta \E \left[\int_0^{t\land \tau_\eta} M_s^\beta ds \right]\\
&\leq & 
C_\beta \int_0^t \E\left[ M_{s\land \tau_\eta}^\beta \right] ds \leq C_{\eta,\beta} 
t^2.
\end{eqnarray*}
Next, $\beta/\gamma(M_s) \geq \beta/(\gamma(0)+\eta)>1$ on $[0,\tau_\eta)$,
whence (since $2^{-ax}\leq C_a x$ for all $x>0$),
\begin{eqnarray*}
J_t^{\beta,\eta} &\leq &  C_{\eta,\beta}\E \left[\int_0^{t\land \tau_\eta}  
2^{[1-\beta/(\gamma(0)+\eta)]/\kappa(M_s)}
ds \right] \leq  C_{\eta,\beta}\E \left[\int_0^{t\land \tau_\eta} \kappa(M_s) 
ds \right] \\
&\leq& C_{\eta,\beta}\E \left[\int_0^{t\land \tau_\eta}  M_s 
ds \right] \leq C_{\eta,\beta} \int_0^t \E\left[ M_{s\land \tau_\eta} \right] ds 
\leq C_{\eta,\beta} t^2.
\end{eqnarray*}
Again, we used here that $\kappa(0)=0$ and  that
$\kappa$ is Lipschitz-continuous. As a conclusion, 
$\E[\Delta_{t\land \tau_\eta}^\beta ]\leq C_{\eta,\beta} t^2$.

We may now conclude that for all $\delta>0$, all $\alpha>0$, all $\eta>0$,
all $\beta \in (\gamma(0)+\eta,1]$,
\begin{eqnarray*}
\mathbb{P} [\alpha^{-1/\gamma(0)}\Delta_{\alpha} \geq \delta] &\leq & 
\mathbb{P} [\tau_\eta\geq \alpha] + \mathbb{P} [\Delta_{\alpha \land \tau_\eta}
\geq \delta \alpha^{1/\gamma(0)}]\\
&\leq & C_{\eta} \alpha 
+ (\delta \alpha^{1/\gamma(0)})^{-\beta}E[ \Delta_{\alpha \land \tau_\eta}^\beta]
\leq  C_{\eta} \alpha + C_{\eta,\delta,\beta} \alpha^{2-\beta/\gamma(0)}.
\end{eqnarray*}
Choosing $\eta=\min(\gamma(0),1-\gamma(0))/2$ and then 
$\beta \in (\gamma(0)+\eta,1\land 2\gamma(0))$,
we deduce that  $2-\beta/\gamma(0)>0$, so that 
$\alpha^{-1/\gamma(0)}\Delta_{\alpha T}$ tends to $0$ in probability when $\alpha$ goes to 0, 
which was our aim.
\end{proof}

%%%%%%%%%%%%%%%%%%%%%%%%%%%%%%%%%%%%%%%%%
%%%%%%%%% Biblio %%%%%%%%%%%%%%%%%%%%%%%%
%%%%%%%%%%%%%%%%%%%%%%%%%%%%%%%%%%%%%%%%%

\end{document}